\documentclass [twoside,reqno, 12pt] {amsart}

\usepackage{amsfonts}
\usepackage{amssymb}
\usepackage{a4}
\usepackage{color}

\newtheorem{thm}{Theorem}[section]
\newtheorem{cor}[thm]{Corollary}

\newtheorem{prop}[thm]{Proposition}

\newtheorem{defn}[thm]{Definition}
\newtheorem{rem}[thm]{Remark}

\theoremstyle{definition}

\numberwithin{equation}{section}

\renewcommand{\Re}{\hbox{Re}\,}
\renewcommand{\Im}{\hbox{Im}\,}

\newcommand{\C}{\mathbb{C}}

\newcommand{\N}{\mathbb{N}}

\newcommand{\R}{\mathbb{R}}

\newcommand{\supp}{\operatorname{supp}}

\def\hat{\widehat}
\def\tilde{\widetilde}
\def \bfo {\begin {eqnarray*} }
\def \efo {\end {eqnarray*} }
\def \ba {\begin {eqnarray*} }
\def \ea {\end {eqnarray*} }
\def \beq {\begin {eqnarray}}
\def \eeq {\end {eqnarray}}
\def \supp {\hbox{supp }}

\def \p {\partial}

\def\hat{\widehat}
\def\tilde{\widetilde}
\def \bfo {\begin {eqnarray*} }
\def \efo {\end {eqnarray*} }
\def \ba {\begin {eqnarray*} }
\def \ea {\end {eqnarray*} }
\def \beq {\begin {eqnarray}}
\def \eeq {\end {eqnarray}}
\def \supp {\hbox{supp }}

\def \p {\partial}

\begin{document}

 \title[Inverse problems for a magnetic Schr\"odinger operator ]{Inverse problems with partial data for a magnetic Schr\"odinger operator in an infinite slab and on a bounded domain}

\author[Krupchyk]{Katsiaryna Krupchyk}

\address
        {K. Krupchyk, Department of Mathematics and Statistics \\
         University of Helsinki\\
         P.O. Box 68 \\
         FI-00014   Helsinki\\
         Finland}

\email{katya.krupchyk@helsinki.fi}

\author[Lassas]{Matti Lassas}

\address
        {M. Lassas, Department of Mathematics and Statistics \\
         University of Helsinki\\
         P.O. Box 68 \\
         FI-00014   Helsinki\\
         Finland}

\email{matti.lassas@helsinki.fi}

\author[Uhlmann]{Gunther Uhlmann}

\address
       {G. Uhlmann, Department of Mathematics\\
       University of Washington\\
       Seattle, WA  98195-4350\\
           and\\
       Department of Mathematics\\
       340 Rowland Hall University of California\\
       Irvine, CA 92697-3875\\
       USA}
       
\email{gunther@math.washington.edu}

\maketitle

\begin{abstract} 

In this paper we study inverse boundary value problems with partial data for the magnetic Schr\"odinger operator. In the case of an infinite slab in $\R^n$, $n\ge 3$, 
we establish that the magnetic field and the electric potential can be determined uniquely, when the Dirichlet and Neumann data are given either on the different boundary hyperplanes of the slab or on the same hyperplane.  This is a 
generalization of the results of \cite{LiUhl2010}, obtained for the Schr\"odinger operator  without magnetic potentials. 

In the case of a bounded domain in $\R^n$, $n\ge 3$, extending the results of \cite{Ammari_Uhlmann_2004}, we show the unique determination of  the magnetic field and electric potential from the Dirichlet and Neumann data,  given on two arbitrary open subsets of the boundary, provided that  the magnetic and electric potentials are known in a neighborhood of the boundary.  Generalizing the results of \cite{Isakov_2007}, we also obtain uniqueness results for the magnetic Schr\"odinger operator, when the Dirichlet and Neumann data are known on the same part of the boundary, assuming that the inaccessible part of the boundary is a part of a hyperplane.

\end{abstract}

\section{Introduction and statement of results}

The purpose of this paper is to study inverse boundary value problems with partial data for the magnetic Schr\"odinger operator
 on a bounded domain in $\R^n$, $n\ge 3$, as well as in an infinite slab in $\R^n$. 

We shall start by discussing the case of the slab.  Let $\Sigma\subset\R^n$, $n\ge 3$, be an infinite slab between two parallel hyperplanes $\Gamma_1$ and $\Gamma_2$. 
Without loss of generality, we shall assume that 
\[
\Sigma=\{x=(x',x_n)\in\R^n:x'=(x_1,\dots,x_{n-1})\in \R^{n-1},0<x_n<L\}, \quad L>0,
\]
and 
\[
\Gamma_1=\{x\in\R^n:x_n=L\},\quad \Gamma_2=\{x\in\R^n:x_n=0\}. 
\]
Consider the magnetic Schr\"odinger operator
\[  
\mathcal{L}_{A,q}(x,D)=\sum_{j=1}^n(D_j+A_j(x))^2+q(x),
\]
with magnetic potential $A=(A_j)_{1\le j\le n}\in W^{1,\infty}(\Sigma,\C^n)$ and electric potential $q\in L^\infty(\Sigma,\C)$. Here $D=i^{-1}\nabla$. 
In what follows, we shall  assume that $A$ and $q$ are compactly supported. 
According to Proposition \ref{prop_essential_spec} in Appendix A, the operator $\mathcal{L}_{A,q}(x,D)$, equipped with the domain $H^1_0(\Sigma)\cap H^2(\Sigma)$ is closed and its essential spectrum is equal to $[\pi^2/L^2,+\infty)$.

We shall be concerned with the following Dirichlet problem,
\begin{equation}
\label{eq_Dirichlet_problem}
\begin{aligned}
(\mathcal{L}_{A,q}(x,D)-k^2)u(x)&=0\quad \textrm{in}\quad \Sigma,\\
u&=f\quad\textrm{on}\quad \Gamma_1,\\
u&=0\quad\textrm{on}\quad \Gamma_2,
\end{aligned}
\end{equation}
where $k\ge 0$ is fixed and  $f\in H^{3/2}(\Gamma_1)$ is with compact support in $\Gamma_1$. When $k<\pi/L$ and $k^2$ avoids the eigenvalues of $\mathcal{L}_{A,q}$, the  problem \eqref{eq_Dirichlet_problem} has a unique solution  $u\in H^2(\Sigma)$. When the spectral parameter $k^2$ is on the essential spectrum of $\mathcal{L}_{A,q}$, to discuss the solvability of the problem  \eqref{eq_Dirichlet_problem},  in Appendix A we introduce the notion of an admissible frequency $k$ and an admissible solution $u$.  Roughly speaking,  the notion of admissibility of a solution $u$  means that a finite number of the Fourier coefficients of $u$ with respect to $x_n$ satisfy the Sommerfeld radiation condition at infinity.   Furthermore, when $A$ and $q$ are real, so that the operator $\mathcal{L}_{A,q}$ is self-adjoint, we show in   Proposition \ref{prop_admiss_fr} that  if $k\ge \pi/L$ is such that $k^2$ avoids the embedded eigenvalues and the set of thresholds $\{(\pi l/L)^2:l=1,2,\dots\}$ of $\mathcal{L}_{A,q}$, then $k$ is admissible for $\mathcal{L}_{A,q}$. 

If $k$ is admissible  for the operator $\mathcal{L}_{A,q}$, we show in Appendix A that  the problem \eqref{eq_Dirichlet_problem} has a unique admissible solution $u$. 
 Notice that $u\in H^2_{\textrm{loc}}(\overline{\Sigma})$, where 
we recall that 
 \[
 H^2_{\textrm{loc}}(\overline{\Sigma})=\{u|_{\Sigma}:u\in H^2_{\textrm{loc}}(\R^n)\}. 
 \]
We define the Dirichlet--to-Neumann map for the magnetic Schr\"odinger operator in the infinite slab $\Sigma$ by
\[
\mathcal{N}_{A,q}: H^{3/2}(\Gamma_1)\cap \mathcal{E}'(\Gamma_1)\to H^{1/2}_{\textrm{loc}}(\p \Sigma),\quad  f\mapsto (\p_{\nu}+iA\cdot\nu) u|_{\p \Sigma},
\]
where $u$ is the solution of \eqref{eq_Dirichlet_problem}.  Here $\nu$ is the unit outer normal to the boundary $\p \Sigma=\Gamma_1\cup\Gamma_2$.

As it was noticed in \cite{Sun_1993}, the Dirichlet--to--Neumann map is invariant under  gauge transformations of the magnetic potential. It follows from the identities  
\begin{equation}
\label{eq_gauge}
e^{-i\Psi}\mathcal{L}_{A,q}e^{i\Psi}=\mathcal{L}_{ A+\nabla\Psi,  q},\quad e^{-i\Psi}\mathcal{N}_{ A, q} e^{i\Psi}=\mathcal{N}_{A+\nabla \Psi,q},
\end{equation}
that  $\mathcal{N}_{A, q} =\mathcal{N}_{A+\nabla \Psi, q}$ when $\Psi\in C^{1,1}(\overline \Sigma)$ compactly supported is such that $\Psi|_{\p\Sigma}=0$.  Thus, $\mathcal{N}_{A,q}$ carries only information about the magnetic field $dA$, where $A$ is viewed as the $1$-form $\Sigma_{j=1}^n A_jdx_j$.  

We shall now state two main results of this paper, which generalize the corresponding results of \cite{LiUhl2010}, obtained in the case of the  Schr\"odinger operator without a magnetic potential. 
The first result, concerning the case when the data and the measurements are on different boundary hyperplanes,  is as follows.

\begin{thm} 
\label{thm_main_1}

Let $\Sigma\subset\R^n$, $n\ge 3$, be an infinite slab between two parallel hyperplanes $\Gamma_1$ and $\Gamma_2$, and let  $A^{(j)}\in W^{1,\infty}(\Sigma,\C^n)\cap \mathcal{E}'(\overline{\Sigma},\C^n)$,  $q^{(j)}\in L^\infty(\Sigma,\C)\cap \mathcal{E}'(\overline{\Sigma},\C)$, $j=1,2$. Denote by $B$ an open ball in $\R^n$, containing the supports of $A^{(j)}$, $q^{(j)}$, $j=1,2$, and let $\gamma_j\subset\Gamma_j$ be arbitrary open sets  such that 
\[
\Gamma_j\cap\overline{B}\subset \gamma_j,\quad j=1,2.
\]
Assume that $k\ge 0$ is admissible  in the sense of  Definition  \emph{\ref{def_adm_k}} for the operator $\mathcal{L}_{A^{(j)},q^{(j)}}$ and its real transpose $\mathcal{L}_{-A^{(j)},q^{(j)}}$, $j=1,2$.   If
\begin{equation}
\label{eq_DN_maps_coincide}
\mathcal{N}_{A^{(1)},q^{(1)}}(f)|_{\gamma_2}=\mathcal{N}_{A^{(2)},q^{(2)}}(f)|_{\gamma_2}, 
\end{equation}
for any $f\in H^{3/2}(\Gamma_1)$, $\supp(f)\subset \gamma_1$, then $dA^{(1)}=dA^{(2)}$ and $q^{(1)}=q^{(2)}$ in $\Sigma$. 

\end{thm}

The assumption that $k\ge 0$ is admissible for the real transpose $\mathcal{L}_{-A^{(j)},q^{(j)}}$ of the operator $\mathcal{L}_{A^{(j)},q^{(j)}}$ is needed  when proving a Runge type approximation result in the infinite slab.  We would also  like to remark that when the operator  $\mathcal{L}_{A^{(j)},q^{(j)}}$  is self-adjoint and $k^2$ is not an eigenvalue and not in the set of thresholds $\{(\pi l/L)^2:l=1,2,\dots\}$ of the operator $\mathcal{L}_{A^{(j)},q^{(j)}}$ , then $k\ge 0$ is admissible for both operators $\mathcal{L}_{A^{(j)},q^{(j)}}$  and $\mathcal{L}_{-A^{(j)},q^{(j)}}$.

Notice that 
if the supports of the coefficients $A^{(j)}$, $q^{(j)}$ are strictly contained in the interior of the slab, then the regions $\gamma_1$ and $\gamma_2$ in Theorem \ref{thm_main_1} can be taken arbitrarily small.  

The next result deals with the inverse problem with the measurements and the data given on the same boundary hyperplane.  

\begin{thm} 
\label{thm_main_2}

Let $\Sigma\subset\R^n$, $n\ge 3$, be an infinite slab between two parallel hyperplanes $\Gamma_1$ and $\Gamma_2$, and let  $A^{(j)}\in W^{1,\infty}(\Sigma,\C^n)\cap \mathcal{E}'(\overline{\Sigma},\C^n)$,  $q^{(j)}\in L^\infty(\Sigma,\C)\cap \mathcal{E}'(\overline{\Sigma},\C)$, $j=1,2$.   Denote by $B$ an open ball in $\R^n$, containing the supports of $A^{(j)}$, $q^{(j)}$, $j=1,2$, and let $\gamma_1,\gamma_1'\subset\Gamma_1$ be arbitrary open sets   such that 
\[
\Gamma_1\cap\overline{B}\subset \gamma_1,\quad \Gamma_1\cap\overline{B}\subset\gamma_1'.
\]
Assume that $k\ge 0$ is admissible  in the sense of  Definition  \emph{\ref{def_adm_k}} for the operator $\mathcal{L}_{A^{(j)},q^{(j)}}$ and its real transpose $\mathcal{L}_{-A^{(j)},q^{(j)}}$, $j=1,2$.
 If 
\[
\mathcal{N}_{A^{(1)},q^{(1)}}(f)|_{\gamma_1'}=\mathcal{N}_{A^{(2)},q^{(2)}}(f)|_{\gamma_1'}, 
\]
for any $f\in H^{3/2}(\Gamma_1)$, $\supp(f)\subset \gamma_1$, then  $dA^{(1)}=dA^{(2)}$ and $q^{(1)}=q^{(2)}$ in $\Sigma$. 

\end{thm}

The main technical tool in proving Theorem \ref{thm_main_1} and Theorem \ref{thm_main_2} is  the construction of  complex geometric optics  solutions \cite{C,SU} with linear phases for the magnetic Schr\"odinger operator, vanishing along a boundary hyperplane.  The idea of constructing such solutions in the case of the Schr\"odinger operator without a magnetic potential,  is based on a  reflection argument and is due to  \cite{Isakov_2007}.  It was applied to the inverse boundary value problem for the Schr\"odinger operator in an infinite slab in the work  \cite{LiUhl2010}, which was our starting point. We would like to emphasize that the case of the Schr\"odinger operator with a magnetic potential is considerably more involved, than the case without magnetic potential, studied in \cite{LiUhl2010}. This is due, in particular, to the fact that a reflection argument with respect to a boundary hyperplane leads to a magnetic potential which is in general only Lipschitz continuous. The construction of complex geometric optics solutions in this case is consequently more complicated, as already seen in  \cite{DKSU_2007} and  \cite{Knu_Salo_2007}.   

When exploiting the complex geometric optics solutions obtained by a reflection argument, we have to control the products of the various phases of the solutions, in the high frequency limit. This leads to some additional constraints 
on the choice of the complex frequency vectors in the phases, which have to be respected when recovering the components of the magnetic field. Notice also that rather than using boundary Carleman estimates in the proof of Theorem \ref{thm_main_1}, as it was done in \cite{LiUhl2010}, here we proceed instead by reflecting both solutions with respect to the different boundary hyperplanes.

Let us consider next physical applications related to Theorems \ref{thm_main_1}  and
\ref{thm_main_2}.
Inverse problems for the Schr\"odinger equation in the slab geometry are encountered in
imaging
of  thin  specimens. A situation analogous to Theorem \ref{thm_main_1}, where sources
are located on one boundary hyperplane
  of  the slab and the field is measured on the other boundary hyperplane, 
 is encountered in the
Transmission Electron Microscopy  (TEM) \cite{M2, M1},  where a beam of electrons is transmitted
through a thin specimen. In TEM
the boundary values on the upper side of the slab are controlled by the electromagnetic
lenses
which manipulate the incoming beam and the  electrons transmitted through the specimen
are detected
below the lower side of the slab.
We note that in TEM with  high energy electrons, 
 the problem is often analyzed using the geometrical optics approximation which leads to
 a problem of integral geometry \cite{M2,Quinto},  but the models
 based directly on the Schr\"odinger equation (see discussion in  \cite[Section 4]{M2})
 are also used.

Situations analogous to Theorem \ref{thm_main_2}, where  the sources are on
the same boundary hyperplane
 of the slab where the fields are detected,  are  also encountered in many electron
microscope applications.
The
  Scanning Tunneling
 Microscope (STM) (see \cite{M3,M2})
and the Dual-tip STM (see \cite{M4})
 are based on the quantum tunneling of electrons between
 a conducting tip (or two conducting tips) and the surface of the  material (i.e.
slab) to be examined.
 If  imaged specimen  is lying on a surface in which electrons cannot propagate,
 the wave function satisfies the Dirichlet boundary
condition on the lower boundary hyperplane. Then, the conducting tips correspond
to both the source and the detection devices, and these measurements
can be modeled using the Dirichlet-to-Neumann map on the  upper boundary hyperplane.

Inverse problems for the Schr\"odinger equation  in a slab
are encountered also in optical tomography \cite{O1}, see the remark at the end of Section 4 for a more detailed discussion.  

Concerning inverse problems in the slab geometry, we would like to mention that apart from \cite{LiUhl2010}, inverse conductivity problems of recovering an unknown embedded object in an infinite slab were studied in \cite{Ikehata_2001, Salo_Wang_2006}, while an inverse scattering problem for the Schr\"odinger operator in a slab was considered in  \cite{CGIF_2005}.

In the remainder of this introduction we shall be concerned with inverse boundary value problems for the magnetic Schr\"odinger operator on a bounded domain.   
Let  $\Omega\subset\R^n$, $n\ge 3$, be a bounded domain with $C^\infty$ boundary.  Consider the following Dirichlet problem,
\begin{equation}
\label{eq_Dir_bounded}
\begin{aligned}
\mathcal{L}_{A,q}u&=0\quad \textrm{in} \quad \Omega,\\
u|_{\p \Omega}&=f,
\end{aligned}
\end{equation}
with $A\in W^{1,\infty}(\Omega,\C^n)$,  $q\in L^\infty(\Omega,\C)$, and  $f\in H^{3/2}(\p \Omega)$. 

The magnetic Schr\"odinger operator $\mathcal{L}_{A,q}$ in $L^2(\Omega)$, equipped with the domain $H^2(\Omega)\cap H^1_0(\Omega)$, is closed with the discrete spectrum. 
Let us make the following assumption:
\begin{itemize}
\item[\textbf{(A)}] $0$ is not an eigenvalue of the magnetic Schr\"odinger operator $\mathcal{L}_{A,q}:H^2(\Omega)\cap H^1_0(\Omega)\to L^2(\Omega)$.
\end{itemize}
Under the assumption (A), the Dirichlet problem \eqref{eq_Dir_bounded} has a unique solution $u\in H^2(\Omega)$, and we can introduce the Dirichlet--to--Neumann map
\[
\Lambda_{A,q}: H^{3/2}(\p \Omega)\to H^{1/2}(\p \Omega),\quad  f\mapsto (\p_{\nu}+iA\cdot\nu) u|_{\p \Omega},
\]
where $\nu$ is the unit outer normal to the boundary. 

Let $\gamma_1,\gamma_2\subset \p\Omega$ be non-empty open subsets of the boundary. We are interested in the inverse boundary value problem for the operator $\mathcal{L}_{A,q}$ with partial boundary measurements: assuming that 
\[
\Lambda_{A^{(1)},q^{(1)}}(f)|_{\gamma_2}=\Lambda_{A^{(2)},q^{(2)}}(f)|_{\gamma_2}, 
\]
for all $f\in H^{3/2}(\p\Omega)$, $\supp(f)\subset \gamma_1$, can we conclude that $dA^{(1)}=dA^{(2)}$ and $q^{(1)}=q^{(2)}$ in $\Omega$?

When measurements are done on the entire boundary,  inverse 
problems for  various second order elliptic
equations have been studied e.g.\ in \cite{AP,Buk,GLU1,N2,N1,PPU}.
For very non-regular coefficient functions there are counterexamples to
the uniqueness of the inverse problems
\cite{GLU3, GKLUbull} which are closely related to the so-called invisibility
cloaking \cite{GKLU1,GKLUsiam,GKLU8}.

Now in many applications,  performing measurements on the entire boundary could be either impossible or too cost consuming.   Therefore, the inverse boundary value problem with partial measurements, formulated above, is both natural and important, see e.g.\ \cite{ALP,KKL,KenSjUhl2007,LU,LTU,LeU} for related problems.   
To the best of our knowledge, the partial data problem still remains open in general, even in the absence of a magnetic potential.   In this case, under the assumption that $q^{(1)}=q^{(2)}$ in a neighborhood of the boundary of $\Omega$, the problem was settled in \cite{Ammari_Uhlmann_2004}.   Dropping this assumption, it was shown in \cite{BukhUhl_2002} that the electric potential can be  uniquely determined by the Dirichlet--to--Neumann map when $\gamma_1=\p \Omega$ and $\gamma_2$ is, roughly speaking, a half of the boundary.  In \cite{KenSjUhl2007}, this result was significantly improved and it was shown that $\gamma_2$ can be possibly very small, while it is still required that $\gamma_1$ and $\gamma_2$ should have a non-void intersection.  On the other hand, for special geometries of the domain, in \cite{Isakov_2007}, the identifiability result was established when $\gamma_1=\gamma_2$ is such that the remaining part of the boundary is contained in a hyperplane or a sphere. 

In the presence of a magnetic potential, the inverse problem of determining the magnetic field and the electric potential from partial boundary measurements was addressed in 
\cite{DKSU_2007}, when $\gamma_1=\p \Omega$ and $\gamma_2$ is possibly a very small subset of the boundary,  see also \cite{Knu_Salo_2007}. 
Under the assumption that $A^{(1)}=A^{(2)}$ and $q^{(1)}=q^{(2)}$ in a neighborhood of the boundary, in \cite{Ben_Joud} it is proven that the magnetic field and the electric potential can be uniquely determined by boundary measurements, provided that $\gamma_1=\p \Omega$ and $\gamma_2$ is arbitrary. Logarithmic stability estimates for this problem are also obtained in \cite{Ben_Joud}.

Under the assumption that $A^{(1)}=A^{(2)}$ and $q^{(1)}=q^{(2)}$ in a neighborhood of the boundary, generalizing the work \cite{Ammari_Uhlmann_2004}, we have the following simple result. 

\begin{thm}
\label{thm_AU}
Let $\Omega\subset \R^n$, $n\ge 3$,  be a bounded domain with $C^\infty$ connected boundary, and $A^{(j)}\in W^{1,\infty}(\Omega,\C^n)$ and $q^{(j)}\in L^\infty(\Omega,\C)$, $j=1,2$, be such that the assumption \emph{(A)} is satisfied for both operators.  Assume that 
$A^{(1)}=A^{(2)}$ and $q^{(1)}=q^{(2)}$ in a  neighborhood of the boundary $\p \Omega$. 
Let   $\gamma_1,\gamma_2\subset \p\Omega$ be   non-empty  open subsets of the boundary. If 
\[
\Lambda_{A^{(1)},q^{(1)}}(f)|_{\gamma_2}=\Lambda_{A^{(2)},q^{(2)}}(f)|_{\gamma_2},
\]
 for all $f\in H^{3/2}(\p\Omega)$, $\supp(f)\subset \gamma_1$, then $dA^{(1)}=dA^{(2)}$ and $q^{(1)}=q^{(2)}$ in $\Omega$.
\end{thm}

In Theorem \ref{thm_AU}, the supports of $A^{(1)}-A^{(2)}$ and $q^{(1)}-q^{(2)}$ are not allowed to  come close to the boundary of $\Omega$. However, 
this condition can be weakened for special bounded domains, say, for domains of the form $\Omega=\omega\times [0,L]$. Here $\omega\subset \R^{n-1}$ is an open bounded domain in $\R^{n-1}$ with connected smooth boundary.  Assume that $A^{(1)}=A^{(2)}$ and $q^{(1)}=q^{(2)}$ near $\p \omega\times [0,L]$. 
If 
\[
\Lambda_{A^{(1)},q^{(1)}}(f)|_{\omega\times\{0\}}=\Lambda_{A^{(2)},q^{(2)}}(f)|_{\omega\times\{0\}},
\]
 for all $f\in H^{3/2}(\p(\omega\times[0,L]))$, $\supp(f)\subset \omega\times\{L\}$, then $dA^{(1)}=dA^{(2)}$ and $q^{(1)}=q^{(2)}$ in $\omega\times[0,L]$.
Notice  in particular that supports of $A^{(j)}$ and $q^{(j)}$ can approach the flat parts of the boundary of the cylinder, $\omega\times\{0\}$ and $\omega\times\{L\}$. 

This observation follows from the proof of 
Theorem \ref{thm_main_1}, when a ball $B\subset \R^n$ is replaced by a cylinder $\omega'\times[0,L]$, where $\omega'\subset\subset \omega$ is a domain in $\R^{n-1}$ with  connected smooth boundary, such that $\supp(A^{(1)}-A^{(2)}),\supp(q^{(1)}-q^{(2)})\subset \omega'\times [0,L]$.

Finally, we have the following generalization of a result from \cite{Isakov_2007} to the case of the magnetic Schr\"odinger operator, where the Dirichlet and Neumann data are known on the same part of the boundary, assuming that the inaccessible part of the boundary is a part of a hyperplane.

\begin{thm}
\label{thm_Isak}
Let $\Omega\subset\{\R^n:x_n>0\}$, $n\ge 3$,  be a bounded domain with connected $C^\infty$ boundary, and  let $\gamma_0=\p \Omega\cap \{x_n=0\}\ne\emptyset$ and $\gamma=\p \Omega\setminus\gamma_0$.  Let $A^{(j)}\in W^{1,\infty}(\Omega,\C^n)$ and $q^{(j)}\in L^\infty(\Omega,\C)$, $j=1,2$,  be such that the assumption \emph{(A)} is satisfied for both operators. If 
\[
\Lambda_{A^{(1)},q^{(1)}}(f)|_{\gamma}=\Lambda_{A^{(2)},q^{(2)}}(f)|_{\gamma}, 
\]
for any $f\in H^{3/2}(\p\Omega)$, $\supp(f)\subset \overline{\gamma}$, then $dA^{(1)}=dA^{(2)}$ and $q^{(1)}=q^{(2)}$ in $\Omega$. 

\end{thm}

The plan of the paper is as follows. In Section 2, we review the construction of complex geometric optics solutions for the magnetic Schr\"odinger operator with a Lipschitz continuous magnetic potential, following \cite{Knu_Salo_2007}. 
Section 3 is devoted to the proof of Theorem \ref{thm_main_1}, while the proof of Theorem \ref{thm_main_2} is given in Section 4. Theorems \ref{thm_AU} and \ref{thm_Isak}, concerned with the case of bounded domains, are established in Section 5.  Appendix A describes the construction of admissible solutions to the Dirichlet problem \eqref{eq_Dirichlet_problem} 
 in an infinite slab, considered in the main part of the paper.

\section{Complex geometric optics solutions}

When proving Theorems \ref{thm_main_1} and  \ref{thm_main_2}, we shall employ a reflection argument across the boundary hyperplanes, which will lead to the magnetic potentials which are Lipschitz continuous on the extended domain.  To this end, we shall start by recalling a construction of  complex geometric optics solutions for the magnetic Schr\"odinger operator under these limited regularity assumptions. Here we follow the works \cite{DKSU_2007} and particularly, \cite{Knu_Salo_2007}.

Let $\Omega\subset \R^n$, $n\ge 3$, be a bounded domain with $C^\infty$-boundary.  Consider the magnetic Schr\"odinger equation,
\begin{equation}
\label{eq_Sch_lip}
\mathcal{L}_{A,q}u=0\quad \textrm{in}\quad \Omega, 
\end{equation}
where $A\in W^{1,\infty}(\Omega,\C^n)$ and $q\in L^\infty(\Omega,\C)$.  Following \cite{Knu_Salo_2007}, we recall the construction of complex geometric optics solutions
\begin{equation}
\label{eq_cgo_lip}
u(x,\zeta;h)=e^{x\cdot\zeta/h}(a(x,\zeta;h)+r(x,\zeta;h))
\end{equation}
of \eqref{eq_Sch_lip},  which is based on  Carleman estimates and a smoothing argument.  Here $\zeta\in \C^n$, $\zeta\cdot\zeta=0$, $|\zeta|\sim 1$,  $a$ is a smooth amplitude,  $r$ is a correction term, and $h>0$ is a small parameter. 

To deal with the magnetic potential $A\in W^{1,\infty}(\Omega, \C^n)$,  we extend $A$ to a Lipschitz vector field, compactly supported  in $\tilde \Omega$, where $\tilde \Omega\subset \R^n$ is an open bounded set such that $\Omega\subset\subset \tilde \Omega$. We consider the mollification $A^\sharp=A*\varphi_\varepsilon\in C_0^\infty(\tilde \Omega,\C^n)$. Here $\varepsilon>0$ is small and 
$\varphi_\varepsilon(x)=\varepsilon^{-n}\varphi(x/\varepsilon)$ is the usual mollifier with $\varphi\in C^\infty_0(\R^n)$, $0\le \varphi\le 1$, and 
$\int \varphi dx=1$.  We write 
$A^\flat=A-A^\sharp$. We have the following estimates
\begin{equation}
\label{eq_flat_est}
\|A^\flat\|_{L^\infty}=\mathcal{O}(\varepsilon),
\end{equation}
\[
 \|\p^\alpha A^\sharp\|_{L^\infty}=\mathcal{O}(\varepsilon^{-|\alpha|}) \quad \textrm{for all}\quad \alpha,
\]
as $\varepsilon\to 0$. 

In this paper we shall work with $\zeta$ depending slightly on $h$, i.e. $\zeta=\zeta^{(0)}+\zeta^{(1)}$ with $\zeta^{(0)}$ being independent of $h$ and $\zeta^{(1)}=\mathcal{O}(h)$ as $h\to 0$.  Consider the conjugated operator
\begin{align*}
e^{-x\cdot\zeta/h}h^2\mathcal{L}_{A,q} e^{x\cdot\zeta/h}=&-h^2\Delta +2(-i\zeta^{(0)}+hA)\cdot hD-2i\zeta^{(1)}\cdot hD+h^2A^2\\
&-2ih\zeta^{(0)}\cdot (A^\sharp+A^\flat)-2ih\zeta^{(1)}\cdot A 
+h^2(D\cdot A)+h^2q. 
\end{align*}
Then in order that \eqref{eq_cgo_lip} be a solution  of \eqref{eq_Sch_lip}, we need to have
\begin{equation}
\label{eq_transport_Lip}
\zeta^{(0)}\cdot D a+\zeta^{(0)}\cdot A^\sharp a=0 \quad\textrm{in}\quad \Omega,
\end{equation}
\begin{equation}
\label{eq_r}
e^{-x\cdot\zeta/h}h^2\mathcal{L}_{A,q} e^{x\cdot\zeta/h}r=-h^2\mathcal{L}_{A,q}a +2ih\zeta^{(0)}\cdot A^\flat a+2i\zeta^{(1)}\cdot hDa+2ih\zeta^{(1)}\cdot A a\quad\textrm{in}\quad \Omega.
\end{equation}
The equation \eqref{eq_transport_Lip} is the first transport equation and it follows from \cite[Lemma 6.1]{Knu_Salo_2007} that it has a solution $a\in C^\infty(\overline{\Omega})$ which satisfies
\begin{equation}
\label{eq_ampl_est}
\|\p^\alpha a\|_{L^\infty(\Omega)}\le C_\alpha \varepsilon^{-|\alpha|}\quad \textrm{for all}\quad \alpha. 
\end{equation}
The estimate \eqref{eq_ampl_est} follows from the explicit formula for the solution $a=e^{\Phi}$, where $\Phi\in C^\infty(\overline{\Omega})$ is given by
\begin{equation}
\label{eq_expl_Phi}
\begin{aligned}
&\Phi(x,\zeta^0;h)\\
&=\frac{-i}{2\pi}\int_{\R^2}\frac{\zeta^{(0)}\cdot A^\sharp(x-y_1\Re\zeta^{(0)}-y_2\Im\zeta^{(0)})\chi(x-y_1\Re\zeta^{(0)}-y_2\Im\zeta^{(0)})}{y_1+iy_2}dy_1dy_2,
\end{aligned}
\end{equation}
where $\chi\in C^\infty_0(\tilde \Omega)$ is such that $\chi=1$ near $\overline{\Omega}$. 

Using \eqref{eq_ampl_est}, \eqref{eq_flat_est}, and the fact that $\zeta^{(1)}=\mathcal{O}(h)$, for the right hand side of \eqref{eq_r}, we have the following estimate,
\[
\|-h^2\mathcal{L}_{A,q}a +2ih\zeta\cdot A^\flat a+2i\zeta^{(1)}\cdot hDa+2ih\zeta^{(1)}\cdot A a\|_{L^\infty(\Omega)}\le \mathcal{O}(h^2\varepsilon^{-2}+h\varepsilon).
\]
It follows from \cite[Proposition 4.3]{Knu_Salo_2007}  that for $h$ small enough, there is a solution $r\in H^1(\Omega)$ of \eqref{eq_r}, which satisfies $\|r\|_{H^1_{\textrm{scl}}(\Omega)}=\mathcal{O}(h\varepsilon^{-2}+\varepsilon)$.  Here $\|r\|_{H^1_{\textrm{scl}}(\Omega)}=\|r\|_{L^2(\Omega)}+\|h\nabla r\|_{L^2(\Omega)}$.
The optimal choice of $\varepsilon$ is given by $\varepsilon=h^{1/3}$.  
We have therefore the following result, see \cite[Proposition 4.3]{Knu_Salo_2007}.

\begin{prop}
\label{prop_CGO_Lip}
Let $A\in W^{1,\infty}(\Omega,\C^n)$ and $q\in L^\infty(\Omega,\C)$. Then for $h>0$ small enough, there is a solution $u\in H^1(\Omega)$, given by \eqref{eq_cgo_lip}, of the equation \eqref{eq_Sch_lip}, where  $a\in C^\infty(\overline{\Omega})$ solves the transport equation \eqref{eq_transport_Lip}, and satisfies the estimate 
$\|\p^\alpha a\|_{L^\infty(\Omega)}\le C_\alpha h^{-|\alpha|/3}$, and $\|r\|_{H^1_{\emph{\textrm{scl}}}(\Omega)}=\mathcal{O}(h^{1/3})$. 
\end{prop}

\begin{rem}
\label{rem_com_geom_1} In what follows, we shall need complex geometric optics solutions belonging to $H^{2}(\Omega)$. To obtain such solutions, let $\Omega'\supset\supset\Omega$ be a bounded domain with smooth boundary,  and let us extend $A\in W^{1,\infty}(\Omega,\C^n)$ and $q\in L^\infty(\Omega)$ to $W^{1,\infty}(\Omega',\C^n)$ and $L^\infty(\Omega')$-functions, respectively. By elliptic regularity, the complex geometric optics solutions, constructed on $\Omega'$, according to Proposition \emph{\ref{prop_CGO_Lip}},   belong to  $H^{2}(\Omega)$.

\end{rem}

\begin{rem}
\label{rem_CGO_Lip}
Using \eqref{eq_expl_Phi} and \eqref{eq_flat_est}, we see that 
\[
\|\Phi(h)- \Phi^{(0)}\|_{L^\infty(\Omega)}\to 0, \quad h\to 0,
\]
where $\Phi^{(0)}$ solves the equation
\[
\zeta^{(0)}\cdot \nabla \Phi^{(0)}+i\zeta^{(0)}\cdot A=0\quad \textrm{in}\quad \Omega. 
\]
\end{rem}

In what follows, we shall use the standard notation,
\[
(u,v)_{L^2(\Omega)}=\int_{\Omega}u(x)\overline{v(x)}dx,\quad (u,v)_{L^2(\p \Omega)}=\int_{\p \Omega}u(x)\overline{v(x)}dS,
\]
where $dS$ is the surface measure on the boundary of $\Omega$. 

We recall finally the Green formula for the magnetic Schr\"odinger operator $\mathcal{L}_{A,q}$ on a bounded domain $\Omega\subset \R^n$ with $C^\infty$ smooth boundary, see \cite{DKSU_2007},  
\begin{equation}
\label{eq_Green}
(\mathcal{L}_{A,q}u,v)_{L^2(\Omega)}-(u,\mathcal{L}_{\overline{A},\overline{q}}v)_{L^2(\Omega)}=(u,(\p_\nu+i\nu\cdot\overline{A})v)_{L^2(\p \Omega)}
-((\p_\nu+i\nu\cdot A)u,v)_{L^2(\p \Omega)},
\end{equation}
which is valid for all $u,v\in H^2(\Omega)$.

\section{Proof of Theorem \ref{thm_main_1}}

Assume that  $k\ge 0$ is admissible  for the operator $\mathcal{L}_{A^{(j)},q^{(j)}}$ and its real transpose  $\mathcal{L}_{-A^{(j)},q^{(j)}}$, $j=1,2$. 
Let $u_1\in H^2_{\textrm{loc}}(\overline{\Sigma})$ be the admissible solution to the Dirichlet problem,
\begin{equation}
\label{eq_u_1}
\begin{aligned}
(\mathcal{L}_{A^{(1)},q^{(1)}}(x,D)-k^2)u_1(x)&=0\quad \textrm{in}\quad \Sigma,\\
u_1&=f\quad\textrm{on}\quad \Gamma_1,\\
u_1&=0\quad\textrm{on}\quad \Gamma_2,
\end{aligned}
\end{equation}
for  $f\in H^{3/2}(\Gamma_1)$ such that $\supp(f)\subset \gamma_1$.  Here the existence and uniqueness of an admissible solution is guaranteed by the results of Appendix A.   Let also $v\in H^2_{\textrm{loc}}(\overline{\Sigma})$ be the admissible solution of the following problem,
\begin{align*}
(\mathcal{L}_{A^{(2)},q^{(2)}}(x,D)-k^2)v(x)&=0\quad \textrm{in}\quad \Sigma,\\
v&=u_1\quad\textrm{on}\quad \Gamma_1\cup \Gamma_2.
\end{align*}
Setting $w=v-u_1$, we get
\begin{equation}
\label{eq_w}
\begin{aligned}
(\mathcal{L}_{A^{(2)},q^{(2)}}(x,D)-k^2)w&=(A^{(1)}-A^{(2)})\cdot D u_1+D\cdot ((A^{(1)}-A^{(2)})u_1)\\
&+((A^{(1)})^2-(A^{(2)})^2+q^{(1)}-q^{(2)})u_1\quad \textrm{in }\Sigma.
\end{aligned}
\end{equation}
It follows from \eqref{eq_DN_maps_coincide} that
\[
(\p_{\nu}+iA^{(1)}\cdot\nu) u_1|_{\gamma_2}=(\p_\nu+i A^{(2)}\cdot\nu) v|_{\gamma_2},
\]
and therefore, $\p_{\nu}w=0$ on $\gamma_2$, since $u_1=v=0$ on $\Gamma_2$. 
We denote  
\[
l_1:=\Gamma_1\cap \overline{B}\subset \gamma_1, \quad l_2:=\Gamma_2\cap \overline{B}\subset \gamma_2,\quad l_3:=\p B\cap \Sigma. 
\]
It follows from \eqref{eq_w} that $w\in H^2_{\textrm{loc}}(\overline{\Sigma})$ is a solution to
\[
(-\Delta-k^2)w=0 \quad \textrm{in}\quad \Sigma\setminus{\overline{B}}.
\]
As $w=\p_\nu w=0$ on $\gamma_2\setminus{\overline{l_2}}$, by unique continuation, $w=0$ in $\Sigma\setminus{\overline{B}}$.   
Therefore,  $w=\p_\nu w=0$ on $l_3$.  

Let $u_2\in H^2(\Sigma\cap B)$ be a solution of the equation 
\begin{equation}
\label{eq_u_2}
(\mathcal{L}_{\overline{A^{(2)}},\overline{q^{(2)}}}(x,D)-k^2)u_2=0\quad \textrm{in}\quad \Sigma\cap B,
\end{equation}
such that 
\begin{equation}
\label{eq_u_2_zero}
u_2=0\quad \textrm{on}\quad l_1.
\end{equation}
Then by the Green formula \eqref{eq_Green}, we have
\begin{equation}
\label{eq_cons_green}
\begin{aligned}
((\mathcal{L}_{A^{(2)},q^{(2)}}-k^2)w,u_2)_{L^2(\Sigma\cap B)}=(w,(\mathcal{L}_{\overline{A^{(2)}},\overline{q^{(2)}}}-k^2)u_2)_{L^2(\Sigma\cap B)}\\
+
(w,(\p_\nu+i\nu\cdot\overline{A^{(2)}})u_2)_{L^2(\p (\Sigma\cap B))}-
((\p_{\nu}+i\nu \cdot A^{(2)})w,u_2)_{L^2(\p (\Sigma\cap B))}.
\end{aligned}
\end{equation}
Recall that $\p (\Sigma\cap B)=l_1\cup l_2\cup l_3$, $w=0$ on $\p (\Sigma\cap B)$ and $\p_\nu w=0$ on $l_2\cup l_3$. 
Thus, \eqref{eq_u_2}, \eqref{eq_u_2_zero} and \eqref{eq_cons_green} imply that 
\begin{equation}
\label{eq_cons_green_2}
((\mathcal{L}_{A^{(2)},q^{(2)}}-k^2)w,u_2)_{L^2(\Sigma\cap B)}=0. 
\end{equation}
Using \eqref{eq_w} and \eqref{eq_cons_green_2}, we get 
\begin{equation}
\label{eq_cons_green_3}
\begin{aligned}
\int_{\Sigma\cap B} (A^{(1)}-A^{(2)})\cdot ((Du_1)\overline{u_2}+u_1\overline{Du_2})dx-i \int_{\p (\Sigma\cap B)}(A^{(1)}-A^{(2)})\cdot\nu u_1\overline{u_2}dS\\
+\int_{\Sigma\cap B} ((A^{(1)})^2-(A^{(2)})^2+q^{(1)}-q^{(2)})u_1\overline{u_2}dx=0.
\end{aligned}
\end{equation}

We may assume without loss of generality that the normal components of $A^{(1)}$ and $A^{(2)}$ are equal to zero on $\Gamma_1\cup\Gamma_2$, i.e.,
\begin{equation}
\label{eq_normal_component_zero}
A^{(1)}\cdot \nu=A^{(2)}\cdot \nu=0\quad \textrm{on}\quad  \Gamma_1\cup\Gamma_2.
\end{equation}
Indeed, it follows from \eqref{eq_gauge} that for $A^{(j)}$,  we can  determine $\Psi^{(j)}\in C^{1,1}(\overline{\Sigma})$ with compact support such that $\Psi^{(j)}|_{\Gamma_1\cup\Gamma_2}=0$ and $\p_\nu \Psi^{(j)}=-A^{(j)}\cdot \nu$ on $\Gamma_1\cup\Gamma_2$, and replace $A^{(j)}$ by $A^{(j)}+\nabla \Psi^{(j)}$. For the existence of such $\Psi^{(j)}\in C^{1,1}(\overline{\Sigma})$, we refer to \cite[Theorem 1.3.3]{Horm_book_1}.

Moreover, by the choice of the set $B$, we have $A^{(1)}=A^{(2)}=0$ on $l_3$.  Thus,
\[
\int_{\p (\Sigma\cap B)}(A^{(1)}-A^{(2)})\cdot\nu u_1\overline{u_2}dS=0,
\]
and therefore, \eqref{eq_cons_green_3} implies that
\begin{equation}
\label{eq_identity_main}
\begin{aligned}
\int_{\Sigma\cap B} (A^{(1)}-A^{(2)})\cdot ((Du_1)\overline{u_2}+u_1\overline{Du_2})dx\\
+\int_{\Sigma\cap B} ((A^{(1)})^2-(A^{(2)})^2+q^{(1)}-q^{(2)})u_1\overline{u_2}dx
=0,
\end{aligned}
\end{equation}
for any $u_1\in W(\Sigma)$ and any $u_2\in V_{l_1}(\Sigma\cap B)$. Here
\begin{align*}
W(\Sigma)=\{&u\in H^2_{\textrm{loc}}(\overline{\Sigma}):(\mathcal{L}_{A^{(1)},q^{(1)}}-k^2)u=0\ \textrm{in}\ \Sigma,\ u|_{\Gamma_2}=0,\ \supp(u|_{\Gamma_1})\subset\gamma_1,\\
&u \textrm{ is admissible in the sense of Appendix A}  \}, 
\end{align*}
\begin{align*}
V_{l_j}(\Sigma\cap  B)=\{&u\in H^2(\Sigma\cap  B): (\mathcal{L}_{\overline{A^{(2)}},\overline{q^{(2)}}}-k^2)u=0\ \textrm{in}\ \Sigma\cap  B,\ u|_{l_j}=0\},
\end{align*}
$j=1,2$.

We would like to replace $u_1$ in \eqref{eq_identity_main} by an element of the space $W_{l_2}(\Sigma\cap B)$, where 
\begin{align*}
W_{l_2}(\Sigma\cap  B)=\{&u\in H^2(\Sigma\cap  B): (\mathcal{L}_{A^{(1)},q^{(1)}}-k^2)u=0\ \textrm{in}\ \Sigma\cap  B,\ u|_{l_2}=0\}.
\end{align*}
 To this end, as in \cite{Ammari_Uhlmann_2004, Isakov_2007,  LiUhl2010}, we need the following Runge type approximation result. 

\begin{prop}
\label{prop_Runge}
The space $W(\Sigma)$ is dense in $W_{l_2}(\Sigma\cap  B)$ in $L^2(\Sigma\cap B)$-topology.
\end{prop}

\begin{proof}
By the Hahn-Banach theorem, we need to show that  for any $g\in L^2(\Sigma\cap B)$ such that
\[
\int_{\Sigma\cap B} g\overline udx=0\quad \textrm{for any } u\in W(\Sigma),
\]
we have
\[
\int_{\Sigma\cap B} g\overline{v}dx=0\quad \textrm{for any } v\in W_{l_2}(\Sigma\cap B). 
\]
Let us extend  $g$ by zero to the complement of $\Sigma\cap B$ in $\Sigma$.  Let $\overline{U}\in H^2_{\textrm{loc}}(\overline{\Sigma})$ be the admissible solution of the problem in the sense of Definition 
\ref{defn_add_sol_app_2},
\begin{align*}
(\mathcal{L}_{-A^{(1)},q^{(1)}}-k^2)\overline{U}&=\overline{g}\quad \textrm{in}\quad \Sigma,\\
\overline{U}&=0\quad\textrm{on}\quad \Gamma_1\cup\Gamma_2.
\end{align*}
Then $U$ solves the equation $(\mathcal{L}_{\overline{A^{(1)}},\overline{q^{(1)}}}-k^2)U=g$ in $\Sigma$. 
For any $u\in W(\Sigma)$, using the Green formula in the infinite slab $\Sigma$, see Proposition \ref{prop_Greens_app}, 
 we have
\begin{align*}
0=\int_{\Sigma} g\overline udx=\int_{\Sigma}[(\mathcal{L}_{\overline{A^{(1)}},\overline{q^{(1)}}}-k^2)U]\overline{u}dx=-\int_{\Gamma_1}\p_\nu U\overline{u}dS.
\end{align*}
Since $u|_{\Gamma_1}$ can be an arbitrary smooth function, supported in  $\gamma_1$, we conclude that $\p_\nu U|_{\gamma_1}=0$. 
Hence, $U$ satisfies the equation 
$(-\Delta-k^2)U=0$ in $\Sigma\setminus B$, and moreover, $U=\p_\nu U=0$ on $\gamma_1\setminus  l_1$. 
Thus, by  unique continuation, $U=0$ in $\Sigma\setminus B$, and  we have $U=\p_\nu U=0$ on $ l_3$.

For any $v\in W_{l_2}(\Sigma\cap B)$, using the Green formula on the bounded domain $\Sigma\cap B$, we have
\begin{align*}
\int_{\Sigma\cap B} g\overline{v}dx=\int_{\Sigma\cap  B}[(\mathcal{L}_{\overline{A^{(1)}},\overline{q^{(1)}}}-k^2)U]\overline{v}dx
=\int_{\Sigma\cap B} U\overline{(\mathcal{L}_{A^{(1)},q^{(1)}}-k^2) v}dx\\
+ \int_{\p (\Sigma\cap  B )}U\overline{(\p_\nu +i\nu\cdot A^{(1)})v}dS-\int_{\p (\Sigma\cap B)}(\p_\nu +i \nu\cdot\overline{A^{(1)}})U\overline{v}dS=0.
\end{align*}
The claim follows. 
\end{proof}

Since $(A^{(1)}-A^{(2)})\cdot \nu=0$ on $\p (\Sigma\cap B)$, we can rewrite \eqref{eq_identity_main} in the following form, 
\begin{align*}
-\int_{\Sigma\cap B} u_1D\cdot((A^{(1)}-A^{(2)})\overline{u_2})dx
+\int_{\Sigma\cap B} (A^{(1)}-A^{(2)})\cdot (u_1\overline{Du_2})dx\\
+\int_{\Sigma\cap B} ((A^{(1)})^2-(A^{(2)})^2+q^{(1)}-q^{(2)})u_1\overline{u_2}dx
=0.
\end{align*}
Hence, an application of  Proposition \ref{prop_Runge} implies that \eqref{eq_identity_main} is valid for any $u_1\in W_{l_2}(\Sigma\cap  B)$ and $u_2\in V_{l_1}(\Sigma\cap B)$.

The next step is to construct complex geometric optics solutions, belonging to the spaces $W_{l_2}(\Sigma\cap  B)$ and $V_{l_1}(\Sigma\cap B)$. Let $\xi,\mu^{(1)},\mu^{(2)}\in\R^n$ be such that $|\mu^{(1)}|=|\mu^{(2)}|=1$ and $\mu^{(1)}\cdot\mu^{(2)}=\mu^{(1)}\cdot\xi=\mu^{(2)}\cdot\xi=0$. 
Similarly to \cite{Sun_1993}, we set 
\begin{equation}
\label{eq_zeta_1_2}
\zeta_1=\frac{ih\xi}{2}+i\sqrt{1-h^2\frac{|\xi|^2}{4}}\mu^{(1)}+\mu^{(2)}, \quad 
\zeta_2=-\frac{ih\xi}{2}+i\sqrt{1-h^2\frac{|\xi|^2}{4}}\mu^{(1)}-\mu^{(2)},
\end{equation}
so that $\zeta_j\cdot\zeta_j=0$, $j=1,2$, and $(\zeta_1+\overline{\zeta_2})/h=i\xi$. Here $h>0$ is a small enough semiclassical parameter.  Moreover, $\zeta_1= i\mu^{(1)}+\mu^{(2)}+\mathcal{O}(h)$ and $\zeta_2= i\mu^{(1)}-\mu^{(2)}+\mathcal{O}(h)$ as $h\to 0$. 

For $u_1$, we need to require that $u_1|_{l_2}=0$.  In order to fulfill this condition, we reflect $\Sigma\cap B$ with respect to the plane $x_n=0$ and denote this reflection by 
$(\Sigma\cap B)_0^*=\{(x',-x_n):x=(x',x_n)\in\Sigma\cap  B\}$. Here $x'=(x_1,\dots,x_{n-1})$.  We also extend the coefficients  $A^{(1)}$ and $q^{(1)}$ to  $(\Sigma\cap B)_0^*$.  For $A^{(1)}_j$, $j=1,\dots,n-1$, and $q^{(1)}$, we do the even extension, and  for $A^{(1)}_n$, we do the odd extension, i.e.,  we set 
\begin{align*}
\tilde A^{(1)}_j(x)&=\begin{cases} A^{(1)}_j(x',x_n),& 0<x_n<L,\\
 A^{(1)}_j(x',-x_n),& -L<x_n<0,
\end{cases},\quad j=1,\dots, n-1,\\
\tilde A^{(1)}_n(x)&=\begin{cases} A^{(1)}_n(x',x_n),& 0<x_n<L,\\
- A^{(1)}_n(x',-x_n),& -L<x_n<0,
\end{cases}\\
\tilde q^{(1)}(x)&=\begin{cases} q^{(1)}(x',x_n),& 0<x_n<L,\\
 q^{(1)}(x',-x_n),& -L<x_n<0.
\end{cases}
\end{align*}  
By \eqref{eq_normal_component_zero}, $A^{(1)}_n|_{x_n=0}=0$, and therefore,  $\tilde A^{(1)}\in W^{1,\infty}((\Sigma\cap  B)\cup (\Sigma\cap  B)_0^*)$ and $\tilde q^{(1)}\in L^{\infty}((\Sigma\cap  B)\cup (\Sigma\cap B)_0^*)$.
Proposition \ref{prop_CGO_Lip} and Remark \ref{rem_com_geom_1} imply that  there exist  complex geometric optics solutions
\[
\tilde u_1(x,\zeta_1;h)=e^{x\cdot \zeta_1/h} (e^{\Phi_1(x,i\mu^{(1)}+\mu^{(2)};h)}+r_1(x,\zeta_1; h))\in H^2((\Sigma\cap B)\cup (\Sigma\cap  B)_0^*)
\] 
of the equation $(\mathcal{L}_{\tilde A^{(1)},\tilde q^{(1)}}-k^2)\tilde u_1=0$ in $(\Sigma\cap B)\cup (\Sigma\cap B)_0^*$, where 
\begin{equation}
\label{eq_r_1}
\|r_1\|_{H^1_{\textrm{scl}}((\Sigma\cap  B)\cup (\Sigma\cap  B)_0^*)}=\mathcal{O}(h^{1/3})
\end{equation}
 and $\Phi_1\in C^{\infty}(\overline{(\Sigma\cap  B)\cup (\Sigma\cap B)_0^*})$ satisfying
\begin{equation}
\label{eq_amplitude_1}
(i\mu^{(1)}+\mu^{(2)})\cdot \nabla e^{\Phi_1}+ i (i\mu^{(1)}+\mu^{(2)})\cdot  (\tilde A^{(1)})^\sharp e^{\Phi_1}=0\quad \textrm{in}\quad (\Sigma\cap  B)\cup (\Sigma\cap B)_0^*,
\end{equation}
\begin{equation}
\label{eq_deriv_Phi_1}
\|\p^{\alpha} e^{\Phi_1}\|_{L^\infty((\Sigma\cap  B)\cup (\Sigma\cap B)_0^*)}\le C_\alpha h^{-|\alpha|/3}, \quad |\alpha|\ge 0. 
\end{equation}
By Remark \ref{rem_CGO_Lip}, $\Phi_1(x,i\mu^{(1)}+\mu^{(2)};h)\to \Phi_1^{(0)}(x,i\mu^{(1)}+\mu^{(2)})$ in the $L^\infty$-norm as $h\to 0$, where $\Phi_1^{(0)}$ solves the equation
\begin{equation}
\label{eq_amplitude_1_without_h}
(i\mu^{(1)}+\mu^{(2)})\cdot \nabla \Phi_1^{(0)}+ i (i\mu^{(1)}+\mu^{(2)})\cdot  \tilde A^{(1)} =0\quad \textrm{in}\quad (\Sigma\cap  B)\cup (\Sigma\cap B)_0^*.
\end{equation}
Let 
\begin{equation}
\label{eq_u_1-CGO}
u_1(x)=\tilde u_1(x',x_n)-\tilde u_1(x',-x_n),\quad x\in \Sigma\cap B. 
\end{equation}
Then it is easy to check that  $u_1\in W_{l_2}(\Sigma\cap B)$.

To construct $u_2$, we have to fulfill the condition $u_2|_{l_1}=0$.  To this end, we reflect $\Sigma\cap B$ with respect to the plane $x_n=L$ and denote this reflection by 
$(\Sigma\cap B)_L^*=\{(x',-x_n+2L):x=(x',x_n)\in\Sigma\cap  B\}$. For the coefficients  $A^{(2)}_j$, $j=1,\dots,n-1$, and $q^{(2)}$, we do the even extension, and  for $A^{(2)}_n$, we do the odd extension, i.e., 
\begin{align*}
\tilde A^{(2)}_j(x)&=\begin{cases} A^{(2)}_j(x',x_n),& 0<x_n<L,\\
 A^{(2)}_j(x',-x_n+2L),& L<x_n<2L,
\end{cases},\quad j=1,\dots, n-1,\\
\tilde A^{(2)}_n(x)&=\begin{cases} A^{(2)}_n(x',x_n),& 0<x_n<L,\\
- A^{(2)}_n(x',-x_n+2L),& L<x_n<2L,
\end{cases}\\
\tilde q^{(2)}(x)&=\begin{cases} q^{(2)}(x',x_n),& 0<x_n<L,\\
 q^{(2)}(x',-x_n+2L),& L<x_n<2L.
\end{cases}
\end{align*}  
As $A^{(2)}_n|_{x_n=L}=0$,  we have  $\tilde A^{(2)}\in W^{1,\infty}((\Sigma\cap  B)\cup (\Sigma\cap  B)_L^*)$ and $\tilde q^{(2)}\in L^{\infty}((\Sigma\cap  B)\cup (\Sigma\cap B)_L^*)$. Thus, by Proposition \ref{prop_CGO_Lip},  one can construct complex geometric optics solutions,
\[
\tilde u_2(x,\zeta_2;h)=e^{x\cdot \zeta_2/h} (e^{\Phi_2(x,i\mu^{(1)}-\mu^{(2)};h)}+r_2(x,\zeta_2; h))\in H^2((\Sigma\cap  B)\cup (\Sigma\cap  B)_L^*)
\]
of the equation $(\mathcal{L}_{\overline{\tilde A^{(2)}},\overline{\tilde q^{(2)}}}-k^2)u_2=0$ in $(\Sigma\cap  B)\cup (\Sigma\cap  B)_L^*$, where 
\begin{equation}
\label{eq_r_2}
\|r_2\|_{H^1_{\textrm{scl}}((\Sigma\cap  B)\cup (\Sigma\cap  B)_L^*)}=\mathcal{O}(h^{1/3})
\end{equation}
 and $\Phi_2\in C^{\infty}(\overline{(\Sigma\cap  B)\cup (\Sigma\cap  B)_L^*})$ satisfies 
\begin{equation}
\label{eq_amplitude_2}
(i\mu^{(1)}-\mu^{(2)})\cdot \nabla \Phi_2+ i (i\mu^{(1)}-\mu^{(2)})\cdot (\overline{\tilde A^{(2)}})^\sharp=0\quad \textrm{in}\quad (\Sigma\cap  B)\cup (\Sigma\cap  B)_L^*,
\end{equation}
\begin{equation}
\label{eq_deriv_Phi_2}
\|\p^\alpha e^{\Phi_2}\|_{L^\infty((\Sigma\cap  B)\cup (\Sigma\cap  B)_L^*)}\le C_\alpha h^{-|\alpha|/3}, \quad |\alpha|\ge 0.
\end{equation}
By Remark \ref{rem_CGO_Lip}, $\Phi_2(x,i\mu^{(1)}-\mu^{(2)};h)\to \Phi_2^{(0)}(x,i\mu^{(1)}-\mu^{(2)})$ in the $L^\infty$-norm as $h\to 0$, where $\Phi_2^{(0)}$ solves the equation
\begin{equation}
\label{eq_amplitude_2_without_h}
(i\mu^{(1)}-\mu^{(2)})\cdot \nabla \Phi_2^{(0)}+ i (i\mu^{(1)}-\mu^{(2)})\cdot \overline{\tilde A^{(2)}}=0\quad \textrm{in}\quad (\Sigma\cap  B)\cup (\Sigma\cap  B)_L^*. 
\end{equation}
Let 
\begin{equation}
\label{eq_u_2-CGO}
u_2(x)=\tilde u_2(x',x_n)-\tilde u_2(x',-x_n+2L),\quad x\in \Sigma\cap B.
\end{equation}
Then $u_2\in V_{l_1}(\Sigma\cap B)$. 

For future references, it will be convenient to have the  following explicit expressions for  the complex geometric optics solutions $u_1$ and $u_2$, given by \eqref{eq_u_1-CGO} and \eqref{eq_u_2-CGO}, 
\begin{equation}
\label{eq_u_1-CGO_1}
u_1(x)=e^{x\cdot\zeta_1/h}(e^{\Phi_1(x)}+r_1(x))-e^{(x',-x_n)\cdot\zeta_1/h}(e^{\Phi_1(x',-x_n)}+r_1(x',-x_n)),
\end{equation}
\begin{equation}
\label{eq_u_2-CGO_2}
u_2(x)=e^{x\cdot\zeta_2/h}(e^{\Phi_2(x)}+r_2(x))-e^{(x',-x_n+2L)\cdot\zeta_2/h}(e^{\Phi_2(x',-x_n+2L)}+r_2(x',-x_n+2L)).
\end{equation}

The next step is to substitute the complex geometric optics solutions $u_1$ and $u_2$ into \eqref{eq_identity_main}. 
To this end, we first note that
\begin{align*}
e^{x\cdot\zeta_1/h}e^{x\cdot\overline{\zeta_2}/h}&=e^{ix\cdot\xi},\\
e^{(x',-x_n)\cdot\zeta_1/h}e^{x\cdot\overline{\zeta_2}/h}&=e^{ix'\cdot\xi'-\frac{2i}{h}\sqrt{1-\frac{h^2|\xi|^2}{4}}\mu^{(1)}_nx_n-2\mu^{(2)}_nx_n/h},\\
e^{x\cdot\zeta_1/h}e^{(x',-x_n+2L)\cdot\overline{\zeta_2}/h}&= e^{ix'\cdot\xi'+2\mu^{(2)}_n(x_n-L)/h+ia_1},\\
e^{(x',-x_n)\cdot\zeta_1/h}e^{(x',-x_n+2L)\cdot\overline{\zeta_2}/h}&=e^{ix'\cdot\xi'-2L\mu^{(2)}_n/h+ia_2},
\end{align*}
where $a_1\in \R$ and $a_2\in \R$ are given by
\begin{align*}
a_1&=2\sqrt{1-\frac{h^2|\xi|^2}{4}}\mu^{(1)}_n\bigg(\frac{x_n}{h}-\frac{L}{h}\bigg)+L\xi_n,\\
a_2&=-x_n\xi_n+L\xi_n-\frac{2L}{h}\sqrt{1-\frac{h^2|\xi|^2}{4}}\mu^{(1)}_n. 
\end{align*}
We shall further assume that $\mu^{(2)}_n>0$ and therefore, for $0<x_n<L$, we have pointwise, 
\begin{equation}
\label{eq_pha}
\begin{aligned}
&|e^{(x',-x_n)\cdot\zeta_1/h}e^{x\cdot\overline{\zeta_2}/h}|\to 0, \ \textrm{as}\ h\to +0, \\
&|e^{x\cdot\zeta_1/h}e^{(x',-x_n+2L)\cdot\overline{\zeta_2}/h}| \to 0, \ \textrm{as}\ h\to +0,\\
&|e^{(x',-x_n)\cdot\zeta_1/h}e^{(x',-x_n+2L)\cdot\overline{\zeta_2}/h}|\to 0, \ \textrm{as}\ h\to +0.
\end{aligned}
\end{equation}
In what follows it will be convenient to write the following norm estimates, which are consequences of  \eqref{eq_deriv_Phi_1}, \eqref{eq_deriv_Phi_2}, \eqref{eq_r_1} and \eqref{eq_r_2}, 
\begin{equation}
\label{eq_rem_amp}
\begin{aligned}
&\|e^{\Phi_j}\|_{L^\infty}=\mathcal{O}(1),\quad \|De^{\Phi_j}\|_{L^\infty}=\mathcal{O}(h^{-1/3}),\\
&\|r_j\|_{L^2}=\mathcal{O}(h^{1/3}),\quad \|Dr_j\|_{L^2}=\mathcal{O}(h^{-2/3}),\quad j=1,2.
\end{aligned}
\end{equation}

For the complex geometric optics solutions $u_1$ and $u_2$, given by \eqref{eq_u_1-CGO_1} and \eqref{eq_u_2-CGO_2},  using \eqref{eq_pha} together with \eqref{eq_rem_amp},
we get 
\[
h\int_{\Sigma\cap B}((A^{(1)})^2-(A^{(2)})^2+q^{(1)}-q^{(2)})u_1\overline{u_2}dx\to 0, \ \textrm{as}\ h\to +0. 
\]
Denoting $\zeta_j^*=(\zeta_j',-(\zeta_j)_n)$ for $\zeta_j=(\zeta_j',(\zeta_j)_n)$, $j=1,2$,  and using  \eqref{eq_u_1-CGO_1} and \eqref{eq_u_2-CGO_2}, we obtain that
\begin{equation}
\label{eq_Du_1}
\begin{aligned}
Du_1(x)=&-\frac{i\zeta_1}{h} e^{x\cdot\zeta_1/h}(e^{\Phi_1(x)}+r_1(x))+ e^{x\cdot\zeta_1/h}(De^{\Phi_1(x)}+Dr_1(x))\\
&+\frac{i\zeta_1^*}{h}
e^{(x',-x_n)\cdot\zeta_1/h}(e^{\Phi_1(x',-x_n)}+r_1(x',-x_n))\\
&-e^{(x',-x_n)\cdot\zeta_1/h}(De^{\Phi_1(x',-x_n)}+Dr_1(x',-x_n)),
\end{aligned}
\end{equation}
\begin{equation}
\label{eq_Du_2}
\begin{aligned}
\overline{Du_2(x)}=&\frac{i\overline{\zeta_2}}{h}e^{x\cdot\overline{\zeta_2}/h}(e^{\overline{\Phi_2(x)}}+\overline{r_2(x)})+e^{x\cdot\overline{\zeta_2}/h}(\overline{De^{\Phi_2(x)}}+\overline{Dr_2(x)})\\
&-\frac{i\overline{\zeta_2^*}}{h}e^{(x',-x_n+2L)\cdot\overline{\zeta_2}/h}(e^{\overline{\Phi_2(x',-x_n+2L)}}+\overline{r_2(x',-x_n+2L)})\\
&-e^{(x',-x_n+2L)\cdot\overline{\zeta_2}/h}(\overline{De^{\Phi_2(x',-x_n+2L)}}+\overline{Dr_2(x',-x_n+L)}).
\end{aligned}
\end{equation}
Using \eqref{eq_u_1-CGO_1},  \eqref{eq_u_2-CGO_2}, \eqref{eq_Du_1} and \eqref{eq_Du_2}, by the dominated convergence theorem together with \eqref{eq_pha} and \eqref{eq_rem_amp}, we get
\begin{align*}
&h\int_{\Sigma\cap B} (A^{(1)}-A^{(2)})\cdot ((Du_1)\overline{u_2}+u_1\overline{Du_2})dx\\
&\to -2i(i\mu^{(1)}+\mu^{(2)})\cdot \int_{\Sigma\cap B} (A^{(1)}-A^{(2)})e^{ix\cdot\xi}e^{\Phi_1^{(0)}(x)+\overline{\Phi_2^{(0)}(x)}}dx,\ \textrm{as}\ h\to+0, 
\end{align*}
where $\Phi_1^{(0)}$ and $\Phi_2^{(0)}$ solve \eqref{eq_amplitude_1_without_h} and \eqref{eq_amplitude_2_without_h}, respectively.

Hence, multiplying \eqref{eq_identity_main} by $h$ and letting $h\to +0$, we obtain 
\begin{equation}
\label{eq_identity_main_2}
(i\mu^{(1)}+\mu^{(2)})\cdot \int_{\Sigma\cap B} (A^{(1)}-A^{(2)})e^{ix\cdot\xi}e^{\Phi_1^{(0)}(x)+\overline{\Phi_2^{(0)}(x)}}dx=0, 
\end{equation}
for all $\xi,\mu^{(1)},\mu^{(2)}\in\R^n$  such that $\mu^{(2)}_n>0$, $|\mu^{(1)}|=|\mu^{(2)}|=1$, and $\mu^{(1)}\cdot\mu^{(2)}=\mu^{(1)}\cdot\xi=\mu^{(2)}\cdot\xi=0$.

In the spirit of \cite{DKSU_2007, Eskin_Rals_1995, Salo_2006,  Sun_1993}, we get the following result. 

\begin{prop}
\label{prop_32}
The equality \eqref{eq_identity_main_2} implies that 
\begin{equation}
\label{eq_identity_main_2_new}
(i\mu^{(1)}+\mu^{(2)})\cdot \int_{\Sigma\cap B} (A^{(1)}-A^{(2)})e^{ix\cdot\xi}dx=0. 
\end{equation}

\end{prop}

\begin{proof}
First notice that it follows from \eqref{eq_amplitude_1_without_h} and \eqref{eq_amplitude_2_without_h} that
\begin{equation}
\label{eq_amplitude_sum}
(i\mu^{(1)}+\mu^{(2)})\cdot\nabla (\Phi_1^{(0)}+\overline{\Phi_2^{(0)}})+i(i\mu^{(1)}+\mu^{(2)})\cdot (A^{(1)}-A^{(2)})=0\quad \textrm{in}\quad \Sigma\cap B. 
\end{equation}
Notice that \eqref{eq_amplitude_1} implies that in the expression \eqref{eq_u_1-CGO_1}  for $u_1$, we may replace $e^{\Phi_1}$ by $ge^{\Phi_1}$ if $g\in C^\infty(\overline{(\Sigma\cap  B)\cup (\Sigma\cap B)_0^*})$ is a solution of 
\begin{equation}
\label{eq_g}
(i\mu^{(1)}+\mu^{(2)})\cdot \nabla g=0\quad  \textrm{in}\quad (\Sigma\cap  B)\cup (\Sigma\cap B)_0^*.  
\end{equation}
Then \eqref{eq_identity_main_2} can be replaced by 
\[
(i\mu^{(1)}+\mu^{(2)})\cdot \int_{\Sigma\cap B} (A^{(1)}-A^{(2)})ge^{ix\cdot\xi}e^{\Phi_1^{(0)}(x)+\overline{\Phi_2^{(0)}(x)}}dx=0.
\]
We conclude from \eqref{eq_amplitude_sum} that
\[
(i\mu^{(1)}+\mu^{(2)})\cdot (A^{(1)}-A^{(2)})ge^{\Phi_1^{(0)}+\overline{\Phi_2^{(0)}}}=i g(i\mu^{(1)}+\mu^{(2)})\cdot\nabla e^{\Phi_1^{(0)}+\overline{\Phi_2^{(0)}}},
\]
and therefore, we have
\begin{equation}
\label{eq_identity_main_3}
\int_{\Sigma\cap B} e^{ix\cdot \xi} g (i\mu^{(1)}+\mu^{(2)})\cdot\nabla e^{\Phi_1^{(0)}+\overline{\Phi_2^{(0)}}} dx=0,
\end{equation}
for all $g$ satisfying \eqref{eq_g}.  

Completing the orthonormal family $\mu^{(2)}$, $\mu^{(1)}$ to an orthonormal basis in $\R^n$, $\mu^{(2)},\mu^{(1)},\mu^{(3)},\dots,\mu^{(n)}$, we have for any vector $x\in \R^n$, 
\[
x=(x\cdot\mu^{(2)})\mu^{(2)}+(x\cdot\mu^{(1)})\mu^{(1)}+ (x\cdot\mu^{(3)})\mu^{(3)}+\cdots+ (x\cdot\mu^{(n)})\mu^{(n)}.
\] 
We introduce new linear coordinates in $\R^n$, given by the orthogonal transformation $T:\R^n\to\R^n$, $T(x)=y$, where $y_1=x\cdot\mu^{(2)}$, $y_2=x\cdot\mu^{(1)}$, $y_j=x\cdot\mu^{(j)}$, $j=3,\dots,n$. 
Denoting $z=y_1+i y_2$ and $\p_{\bar z}=(\p_{y_1}+i\p_{y_2})/2$, we have 
\[
(i\mu^{(1)}+\mu^{(2)})\cdot\nabla =2\p_{\bar z}. 
\]
Thus, changing coordinates in \eqref{eq_identity_main_3}, we get
\begin{equation}
\label{eq_identity_main_4}
\int_{T(\Sigma\cap B)} e^{iy\cdot \xi} g\p_{\bar z}(e^{\Phi_1^{(0)}+\overline{\Phi_2^{(0)}}})dy=0,
\end{equation}
for all $\xi=(0,0,\xi'')$, $\xi''\in\R^{n-2}$, and all $g\in C^\infty(\overline{T(\Sigma\cap B)})$ satisfying $\p_{\bar z} g=0$.  Taking $g=g(z)$ holomorphic in $z$, independent of $y''=(y_3,\dots,y_n)$, and taking the inverse Fourier transform in \eqref{eq_identity_main_4}  in the variable $\xi''$, we get, for all $y''\in\R^{n-2}$,
\[
\int_{T_{y''}}  g(z)\p_{\bar z}(e^{\Phi_1^{(0)}+\overline{\Phi_2^{(0)}}})d\bar z\wedge dz=0,
\]
where $T_{y''}=T(\Sigma\cap B)\cap \Pi_{y''}$ and $\Pi_{y''}=\{(y_1,y_2,y''):(y_1,y_2)\in\R^2\}$. 
Notice that the boundary of $T_{y''}$ is piecewise $C^\infty$-smooth. 
Since 
\[
d(ge^{\Phi_1^{(0)}+\overline{\Phi_2^{(0)}}}dz)=g\p_{\bar z}(e^{\Phi_1^{(0)}+\overline{\Phi_2^{(0)}}})d\bar z\wedge dz,
\]
by the Stokes' formula, we obtain that 
\begin{equation}
\label{eq_orth_hol_g}
\int_{\p T_{y''}} ge^{\Phi_1^{(0)}+\overline{\Phi_2^{(0)}}}dz=0,
\end{equation}
for all holomorphic functions $g\in C^\infty(\overline{T_{y''}})$. 

Next we shall show that \eqref{eq_orth_hol_g} implies that there exists a nowhere vanishing holomorphic function $F\in C(\overline{T_{y''}})$ such that
\begin{equation}
\label{eq_log}
F|_{\p T_{y''}}=e^{\Phi_1^{(0)}+\overline{\Phi_2^{(0)}}}|_{\p T_{y''}}. 
\end{equation}
This follows from the arguments in \cite[Lemma 5.1]{DKSU_2007}.  For the convenience of the reader, we present these arguments here.  
Following \cite[Lemma 5.1]{DKSU_2007}, consider the Cauchy integral   
\[
F(z)=\frac{1}{2\pi i}\int_{\p T_{y''}} \frac{e^{\Phi_1^{(0)}(\zeta)+\overline{\Phi_2^{(0)}(\zeta)}}}{\zeta-z}d\zeta,\quad z\in\C\setminus\p T_{y''}. 
\]
The function $F$ is holomorphic inside and outside of  $\p T_{y''}$.  As $e^{\Phi_1^{(0)}+\overline{\Phi_2^{(0)}}}$ is Lipschitz, the Plemelj-Sokhotski-Privalov formula states that
\begin{equation}
\label{eq_PSP_formula}
\lim_{z\to z_0,z\in T_{y''}} F(z)-\lim_{z\to z_0,z\notin T_{y''}}F(z)=e^{\Phi_1^{(0)}(z_0)+\overline{\Phi_2^{(0)}(z_0)}},\quad z_0\in \p T_{y''}.
\end{equation} 
Since the function $\zeta\mapsto (\zeta-z)^{-1}$ is holomorphic on $T_{y''}$ when $z\notin T_{y''}$, \eqref{eq_orth_hol_g} implies that $F(z)=0$ when $z\notin T_{y''}$.  Hence, the second limit in \eqref{eq_PSP_formula} is zero and therefore, $F$ is holomorphic function on $T_{y''}$ such that \eqref{eq_log} holds. 
Let us show that $F$ nowhere vanishes in $T_{y''}$.  To this end, let  $\p T_{y''}$ be parametrized by $z=\gamma(t)$, and $N$ be the number of zeros of $F$ in $T_{y''}$. Then by the argument principle, 
\[
N=\frac{1}{2\pi i}\int_{z\in \gamma(t)}\frac{F'(z)}{F(z)}dz=\frac{1}{2\pi i}\int_{\zeta\in F(\gamma(t))}\frac{d\zeta}{\zeta}=\frac{1}{2\pi i}\int_{\zeta\in e^{\Phi_1^{(0)}(\gamma(t))+\overline{\Phi_2^{(0)}(\gamma(t))}}}\frac{d\zeta}{\zeta}=0.
\]
The latter equality follows from the fact that the contour $e^{\Phi_1^{(0)}(\gamma(t))+\overline{\Phi_2^{(0)}(\gamma(t))}}$ is homotopic to $\{1\}$ with the homotopy given by $e^{s(\Phi_1^{(0)}(\gamma(t))+\overline{\Phi_2^{(0)}(\gamma(t))})}$, $s\in [0,1]$. The claim follows.

Next since $F$ is nowhere vanishing holomorphic function on $T_{y''}$ and $T_{y''}$ is simply connected, it admits a holomorphic logarithm. Hence,  \eqref{eq_log} implies that 
\[
(\log F)|_{\p T_{y''}}=(\Phi_1^{(0)}+\overline{\Phi_2^{(0)}})|_{\p T_{y''}},
\]
and therefore, by the Cauchy theorem, 
\[
\int_{\p T_{y''}} g (\Phi_1^{(0)}+\overline{\Phi_2^{(0)}}) dz=\int_{\p T_{y''}} g \log F dz=0,
\]
where $g\in C^\infty(\overline{T_{y''}})$ is  an arbitrary function such that $\p_{\bar z}g=0$.  
An application of Stokes' formula gives
\[
\int_{T_{y''}} g\p_{\bar z}(\Phi_1^{(0)}+\overline{\Phi_2^{(0)}})d\bar z\wedge dz=0. 
\]
Taking the Fourier transform with respect to $y''$, we get
\[
\int_{T(\Sigma\cap B)} e^{iy\cdot \xi} g\p_{\bar z}(\Phi_1^{(0)}+\overline{\Phi_2^{(0)}})dy=0,
\]
for all $\xi=(0,0,\xi'')$, $\xi''\in\R^{n-2}$.   Hence, returning back to the variables $x$, we have
\[
(i\mu^{(1)}+\mu^{(2)})\cdot \int_{\Sigma\cap B} e^{i x\cdot \xi}g(x)\nabla(\Phi_1^{(0)}+\overline{\Phi_2^{(0)}})dx=0,
\]
where $g\in C^\infty(\overline{\Sigma\cap B})$ is such that $(i\mu^{(1)}+\mu^{(2)})\cdotÊ\nabla g=0$ in $\Sigma\cap B$. 

Using \eqref{eq_amplitude_sum},  we obtain that
\begin{equation}
\label{eq_identity_main_2_new_zero}
(i\mu^{(1)}+\mu^{(2)})\cdot \int_{\Sigma\cap B} (A^{(1)}-A^{(2)})g(x)e^{ix\cdot\xi}dx=0. 
\end{equation}
With $g=1$, we get \eqref{eq_identity_main_2_new}.   The proof is complete. 

\end{proof}

Since in \eqref{eq_identity_main_2_new} the vector $\mu^{(1)}$ can be replaced by $-\mu^{(1)}$, we get
\begin{equation}
\label{eq_recover_rot}
\mu^{(1)}\cdot \int_{\Sigma\cap B} (A^{(1)}-A^{(2)})e^{ix\cdot\xi}dx=0, 
\end{equation}
for all $\xi,\mu^{(1)}\in\R^n$  such that $\mu^{(1)}\cdot\xi=0$ and for which there is a vector $\mu^{(2)}\in\R^n$ such that $\mu^{(2)}\cdot \mu^{(1)}=\mu^{(2)}\cdot \xi=0$ and $\mu^{(2)}_n>0$.

In the proof of the following result, we shall use some ideas from \cite{Salo_diss}. 

\begin{prop}

\label{prop_33-}
We have 
\begin{equation}
\label{eq_recover_rot_1}
\p_j(A_k^{(1)}-A_k^{(2)})-\p_k(A_j^{(1)}-A_j^{(2)})=0\quad\textrm{in }\Sigma\cap B,\quad 1\le j,k\le n.
\end{equation}

\end{prop}

\begin{proof}
It follows from 
\eqref{eq_recover_rot}   that
\begin{equation}
\label{eq_recover_rot_2}
\mu^{(1)}\cdot (\hat{A^{(1)}\chi_{\Sigma\cap B}}(\xi)-\hat{A^{(2)}\chi_{\Sigma\cap B}}(\xi))=0, 
\end{equation}
where $\chi_{\Sigma\cap B}$ is the characteristic function of the set $\Sigma\cap B$ and $\hat{A^{(j)}\chi_{\Sigma\cap B}}$ stands for the Fourier transform of $A^{(j)}\chi_{\Sigma\cap B}$.  

Let $\xi=(\xi_1,\dots,\xi_n)$, $\xi_j>0$, $j=1,\dots,n$,  and let
\[
\mu^{(1)}(\xi,j,k)=-\xi_k e_j+ \xi_j e_k, \ 1\le j,k\le n,  \ j\ne k,
\]
where $e_1,\dots, e_n$ is the standard orthonormal basis in $\R^n$. 
Then $\mu^{(1)}(\xi,j,k)\cdot\xi=0$. If  $j,k$ are such that $1\le j,k<n$, $j\ne k$, we set 
\[
\mu^{(2)}(\xi,j,k)= -\xi_j\xi_n e_j- \xi_k\xi_n e_k +(\xi_j^2+\xi_k^2)e_n. 
\]
If $k=n$ and $j$ is such that $1\le j<n$, we define
\[
\mu^{(2)}(\xi,j,n)=
 (-\xi_j^2-\xi_n^2)e_l+\xi_l\xi_j e_j + \xi_l\xi_ne_n, 
\]
with some $l\ne j,n$, which exists, since $n\ge 3$. 
 In all cases, we have $\xi\cdot\mu^{(2)}(\xi,j,k)=0$, $\mu^{(1)}(\xi,j,k)\cdot\mu^{(2)}(\xi,j,k)=0$, and $\mu^{(2)}_n(\xi,j,k)>0$. 

Hence, for the vectors $\mu^{(1)}(\xi,j,k)$ and $\xi$, \eqref{eq_recover_rot_2} holds, and it yields that
\[
\xi_j\cdot (\hat{A_k^{(2)}\chi_{\Sigma\cap B}}(\xi)-\hat{A_k^{(1)}\chi_{\Sigma\cap B}}(\xi))-\xi_k\cdot (\hat{A_j^{(2)}\chi_{\Sigma\cap B}}(\xi)-\hat{A_j^{(1)}\chi_{\Sigma\cap B}}(\xi))=0, 
\]
$1\le j,k\le n$, $j\ne k$, for all $\xi\in\R^n$, $\xi_1>0,\dots,\xi_n>0$, and thus, everywhere by the  analyticity of the Fourier transform. 
This proves \eqref{eq_recover_rot_1}.  
\end{proof}

By Proposition \ref{prop_33-}, we have  $dA^{(1)}=dA^{(2)}$ in $\Sigma$.  Since $\Sigma$ is simply connected, there exists $\Psi\in C^{1,1}(\overline{\Sigma})$ with compact support such that 
\begin{equation}
\label{eq_A_1-A_2_nabla}
A^{(1)}-A^{(2)}=\nabla \Psi\quad \textrm{in}\ \Sigma. 
\end{equation}

In particular, $\Psi=0$ along $\p B\cap \Sigma$. 
The next step is to show that $\Psi$ vanishes along the boundary of $\Sigma$.  To this end, substituting \eqref{eq_A_1-A_2_nabla} and $\xi=0$ into 
\eqref{eq_identity_main_2_new_zero}, we get
\begin{equation}
\label{eq_int_m}
(i\mu^{(1)}+\mu^{(2)})\cdot \int_{\Sigma\cap B} (\nabla \Psi) g(x)dx=0,
\end{equation}
where $g\in C^\infty(\overline{\Sigma\cap B})$ is an arbitrary function such that $(i\mu^{(1)}+\mu^{(2)})\cdotÊ\nabla g=0$ in $\Sigma\cap B$. 
We may replace $\mu^{(1)}$ by $-\mu^{(1)}$ in \eqref{eq_int_m}, and  passing to the variables $y$ as in the proof of Proposition \ref{prop_32}, we have
\[
\int_{T(\Sigma\cap B)} g_1(y)\p_{\bar z} \Psi dy=0,\quad \int_{T(\Sigma\cap B)} g_2(y)\p_{z} \Psi dy=0,
\]
where $\p_{\bar z}g_1=0$ and $\p_z g_2=0$. Taking $g_j(y)=g_j'(z)\otimes g_j''(y'')$, $j=1,2$, $y''=(y_3,\dots,y_n)$, and varying $g_j''$ leads to
\[
\int_{T_{y''}} g_1'(z)\p_{\bar z} \Psi d\bar z\wedge dz=0,\quad \int_{T_{y''}} g_2'(z)\p_{z} \Psi dz\wedge d\bar z=0,
\]
where $\p_{\bar z}g_1'=0$ and $\p_z g_2'=0$. Using Stokes' theorem, we get
\[
\int_{\p T_{y''}} g_1'(z) \Psi dz=0,\quad \int_{\p T_{y''}} g_2'(z) \Psi  d\bar z=0. 
\]
In particular, taking $g_2'=\overline{g_1'}$, we have
\[
\int_{\p T_{y''}} \overline{g_1'} \Psi  d\bar z=0,\quad \textrm{and therefore},\quad \int_{\p T_{y''}} g_1' \overline{\Psi}  d z=0. 
\]
Hence, 
\[
\int_{\p T_{y''}} g_1'(z)\Re \Psi  dz=0,\quad \int_{\p T_{y''}} g_1'(z)\Im \Psi  dz=0,
\]
for any holomorphic function $g_1'\in C^\infty(\overline{T_{y''}})$.   Arguing again as in \cite[Lemma 5.1]{DKSU_2007}, we conclude that there exist holomorphic functions $F_j\in C(\overline{T_{y''}})$, $j=1,2$, such that
\[
F_1|_{\p T_{y''}}=\Re \Psi|_{\p T_{y''}},\quad F_2|_{\p T_{y''}}=\Im \Psi|_{\p T_{y''}}. 
\]  
Furthermore, we have
$\Delta\Im F_j=0$ in $T_{y''}$ and $\Im F_j|_{\p T_{y''}}=0$.  Thus, $F_j$ are real-valued and therefore, constant on the connected set $T_{y''}$. Hence, $\Psi$ is constant along $\p T_{y''}$.

Going back to the $x$-coordinates, we conclude that the function $\Psi(x)$ is constant  along  the boundary of the section $T^{-1}(T_{y''})=(\Sigma\cap B)\cap T^{-1}(\Pi_{y''})$,  for all $y''\in \R^{n-2}$, where the two-dimensional plane $T^{-1}(\Pi_{y''})$ is given by
\[
T^{-1}(\Pi_{y''})=\bigg\{x=y_1\mu^{(2)}+y_2\mu^{(1)}+\sum_{j=3}^ny_j\mu^{(j)}:y_1,y_2\in \R,y''=(y_3,\dots,y_n)\bigg\}. 
\]
Here $\mu^{(1)},\mu^{(2)}\in \R^n$ are such that $\mu^{(1)}\cdot\mu^{(2)}=0$, $|\mu^{(1)}|=|\mu^{(2)}|=1$, and $\mu^{(2)}_n>0$.  
Choosing the two-dimensional planes $T^{-1}(\Pi_{y''})$ with $\mu^{(2)}=e_n$, and  $\mu^{(1)}=e_j$, $j=1,\dots,n-1$, and varying $y''$, we conclude that $\Psi$ vanishes along $\p (\Sigma\cap B)$. We refer to \cite{DKSU_2007, KrupLassasUhlmann}
for a detailed discussion in the context of a  general bounded domain.

In order to prove that $q^{(1)}=q^{(2)}$, we may and shall assume that $A^{(1)}=A^{(2)}$. Indeed, as $\Psi$ vanishes along $\Sigma$, it follows from \eqref{eq_A_1-A_2_nabla} and \eqref{eq_gauge} that 
\[
\mathcal{N}_{A^{(1)},q^{(2)}}=\mathcal{N}_{A^{(2)}+\nabla\Psi,q^{(2)}}=\mathcal{N}_{A^{(2)},q^{(2)}},
\] 
and therefore,
\[
\mathcal{N}_{A^{(1)},q^{(1)}}(f)|_{\gamma_2}=\mathcal{N}_{A^{(1)},q^{(2)}}(f)|_{\gamma_2},
\]
for any $f\in H^{3/2}(\Gamma_1)$, $\supp (f)\subset\gamma_1$. 
Substituting $A^{(1)}=A^{(2)}$ in \eqref{eq_identity_main}, we get
\begin{equation}
\label{eq_recovering_q}
\int_{\Sigma\cap B}(q^{(1)}-q^{(2)})u_1\overline{u_2}dx=0. 
\end{equation}
Choosing in \eqref{eq_recovering_q} $u_1$ and $u_2$ being complex geometric optics solutions, given by \eqref{eq_u_1-CGO_1} and \eqref{eq_u_2-CGO_2}, and letting $h\to +0$, we have
\begin{equation}
\label{eq_recovering_q_2}
\int_{\Sigma\cap B}(q^{(1)}-q^{(2)})e^{ix\cdot\xi}e^{\Phi_1^{(0)}(x)+\overline{\Phi_2^{(0)}(x)}}dx=0.
\end{equation}
As before,  notice that \eqref{eq_amplitude_1} implies that in the expression \eqref{eq_u_1-CGO_1}  for $u_1$, we may replace $e^{\Phi_1}$ by $ge^{\Phi_1}$ if $g\in C^\infty(\overline{(\Sigma\cap  B)\cup (\Sigma\cap B)_0^*})$ is a solution of 
\begin{equation}
\label{eq_g_recov}
(i\mu^{(1)}+\mu^{(2)})\cdot \nabla g=0\quad  \textrm{in}\quad (\Sigma\cap  B)\cup (\Sigma\cap B)_0^*.  
\end{equation}
Then \eqref{eq_recovering_q_2} can be replaced by 
\[
\int_{\Sigma\cap B}(q^{(1)}-q^{(2)})g(x)e^{ix\cdot\xi}e^{\Phi_1^{(0)}(x)+\overline{\Phi_2^{(0)}(x)}}dx=0.
\]
Furthermore, \eqref{eq_amplitude_sum} has the form,
\[
(i\mu^{(1)}+\mu^{(2)})\cdot\nabla (\Phi_1^{(0)}+\overline{\Phi_2^{(0)}})=0\quad \textrm{in}\quad \Sigma\cap B. 
\]
Hence, taking $g=e^{-(\Phi_1^{(0)}+\overline{\Phi_2^{(0)}})}$, we get
\[
\int_{\Sigma\cap B}(q^{(1)}-q^{(2)})e^{ix\cdot\xi}dx=0,
\]
for all $\xi\in \R^n$ such that there exist $\mu^{(1)},\mu^{(2)}\in\R^n$, satisfying
\begin{equation}
\label{eq_vect_2}
\xi\cdot \mu^{(1)}=\xi\cdot\mu^{(2)}=\mu^{(1)}\cdot\mu^{(2)}=0,\ |\mu^{(1)}|=|\mu^{(2)}|=1, \ \mu^{(2)}_n>0.
\end{equation}
Let  $\xi=(\xi',\xi_{n-1},\xi_n)\in \R^n$, $\xi'\in\R^{n-2}$, be an arbitrary vector. Then assuming that  $\xi_{n-1}\ne 0$, consider the vector 
\begin{align*}
\tilde{\mu^{(2)}}=\bigg(0_{\R^{n-2}},\frac{-\xi_{n}}{\xi_{n-1}}, 1\bigg),\
\mu^{(2)}=\tilde{\mu^{(2)}}/|\tilde{\mu^{(2)}}|.
\end{align*}
Since $n\ge 3$, there exists $\mu^{(1)}\in\R^n$ such that \eqref{eq_vect_2} holds. Thus, $\hat{q^{(1)}}(\xi)=\hat{q^{(2)}}(\xi)$ for all $\xi\in\R^n$ such that $\xi_{n-1}\ne 0$, and therefore, by continuity of the Fourier transform, for all $\xi\in \R^n$. Hence,  $q^{(1)}=q^{(2)}$ in $\Sigma\cap B$. 
This completes the proof of Theorem \ref{thm_main_1}.

\section{Proof of Theorem \ref{thm_main_2}}

First,  arguing as in the proof of Theorem \ref{thm_main_1}, we obtain the identity \eqref{eq_identity_main}, which  is valid for any $u_1\in W_{l_2}(\Sigma\cap  B)$ and $u_2\in V_{l_2}(\Sigma\cap B)$.

Next we shall construct complex geometric optics solutions, vanishing on $l_2$, using the same choice of complex frequencies $\zeta_1$
and $\zeta_2$, defined in \eqref{eq_zeta_1_2}. The solution $u_1$ will be constructed precisely in the same way as in Theorem \ref{thm_main_1} and it is given by \eqref{eq_u_1-CGO}, see also \eqref{eq_u_1-CGO_1}. 

When constructing $u_2$, we proceed as in the definition of $u_1$ by reflecting the coefficients across the plane $x_n=0$. 
For the coefficients  $A^{(2)}_j$, $j=1,\dots,n-1$, and $q^{(2)}$, we do the even extension, and  for $A^{(2)}_n$, we do the odd extension, 
\begin{align*}
\tilde A^{(2)}_j(x)&=\begin{cases} A^{(2)}_j(x',x_n),& 0< x_n< L,\\
 A^{(2)}_j(x',-x_n),& -L<x_n<0,
\end{cases},\quad j=1,\dots, n-1,\\
\tilde A^{(2)}_n(x)&=\begin{cases} A^{(2)}_n(x',x_n),& 0< x_n<L,\\
- A^{(2)}_n(x',-x_n),& -L< x_n<0,
\end{cases}\\
\tilde q^{(2)}(x)&=\begin{cases} q^{(2)}(x',x_n),& 0< x_n<L,\\
 q^{(2)}(x',-x_n),&-L< x_n<0.
\end{cases}
\end{align*}  
As $A^{(2)}_n|_{x_n=0}=0$,  we have  $\tilde A^{(2)}\in W^{1,\infty}((\Sigma\cap  B)\cup (\Sigma\cap  B)_0^*)$ and $\tilde q^{(2)}\in L^{\infty}((\Sigma\cap  B)\cup (\Sigma\cap B)_0^*)$. 
Then by Proposition \ref{prop_CGO_Lip} and Remark \ref{rem_com_geom_1},  one can construct complex geometric optics solutions,
\[
\tilde u_2(x,\zeta_2;h)=e^{x\cdot \zeta_2/h} (e^{\Phi_2(x,i\mu^{(1)}-\mu^{(2)};h)}+r_2(x,\zeta_2; h))\in H^2((\Sigma\cap  B)\cup (\Sigma\cap  B)_0^*)
\]
of the equation $(\mathcal{L}_{\overline{\tilde A^{(2)}},\overline{\tilde q^{(2)}}}-k^2)u_2=0$ in $(\Sigma\cap  B)\cup (\Sigma\cap  B)_0^*$, where 
\[
%\label{eq_r_2}
\|r_2\|_{H^1_{\textrm{scl}}((\Sigma\cap  B)\cup (\Sigma\cap  B)_0^*)}=\mathcal{O}(h^{1/3}),
\]
 and $\Phi_2\in C^{\infty}(\overline{(\Sigma\cap  B)\cup (\Sigma\cap  B)_0^*})$ satisfying
\begin{equation}
\label{eq_amplitude_2_l_2}
(i\mu^{(1)}-\mu^{(2)})\cdot \nabla \Phi_2+ i (i\mu^{(1)}-\mu^{(2)})\cdot (\overline{\tilde A^{(2)}})^\sharp=0\quad \textrm{in}\quad (\Sigma\cap  B)\cup (\Sigma\cap  B)_0^*,
\end{equation}
\[
\|\p^\alpha e^{\Phi_2}\|_{L^\infty((\Sigma\cap  B)\cup (\Sigma\cap  B)_0^*)}\le C_\alpha h^{-|\alpha|/3},\quad |\alpha|\ge 0.
\]
By Remark \ref{rem_CGO_Lip}, $\Phi_2(x,i\mu^{(1)}-\mu^{(2)};h)\to \Phi_2^{(0)}(x,i\mu^{(1)}-\mu^{(2)})$ in the $L^\infty$-norm as $h\to 0$, where $\Phi_2^{(0)}$ solves the equation
\[
(i\mu^{(1)}-\mu^{(2)})\cdot \nabla \Phi_2^{(0)}+ i (i\mu^{(1)}-\mu^{(2)})\cdot \overline{\tilde A^{(2)}}=0\quad \textrm{in}\quad (\Sigma\cap  B)\cup (\Sigma\cap  B)_0^*. 
\]
Let 
\[
%\label{eq_u_2-CGO}
u_2(x)=\tilde u_2(x',x_n)-\tilde u_2(x',-x_n),\quad x\in \Sigma\cap B.
\]
Then $u_2\in V_{l_2}(\Sigma\cap B)$. 
It will be convenient to have following explicit expression for $u_2$,
\begin{equation}
\label{eq_u_2-CGO_2_l_2}
u_2(x)=e^{x\cdot\zeta_2/h}(e^{\Phi_2(x)}+r_2(x))-e^{(x',-x_n)\cdot\zeta_2/h}(e^{\Phi_2(x',-x_n)}+r_2(x',-x_n)).
\end{equation}

The next step is to substitute complex geometric optics solutions $u_1$ and $u_2$, given by \eqref{eq_u_1-CGO_1} and \eqref{eq_u_2-CGO_2_l_2},  into \eqref{eq_identity_main}. 
To this end, we  first analyze the phases of the products of the complex geometric optics solutions,     
\begin{align*}
e^{x\cdot\zeta_1/h}e^{x\cdot\overline{\zeta_2}/h}&=e^{ix\cdot\xi},\quad e^{(x',-x_n)\cdot\zeta_1/h}e^{(x',-x_n)\cdot\overline{\zeta_2}/h}=e^{i(x',-x_n)\cdot\xi},\\
e^{(x',-x_n)\cdot\zeta_1/h}e^{x\cdot\overline{\zeta_2}/h}&=e^{ix'\cdot\xi'-\frac{2i}{h}\sqrt{1-\frac{h^2|\xi|^2}{4}}\mu^{(1)}_nx_n-2\mu^{(2)}_nx_n/h}=e^{ix\cdot \tilde \xi_--2\mu^{(2)}_nx_n/h},\\
e^{x\cdot\zeta_1/h}e^{(x',-x_n)\cdot\overline{\zeta_2}/h}&= e^{ix'\cdot\xi'+\frac{2i}{h}\sqrt{1-\frac{h^2|\xi|^2}{4}}\mu^{(1)}_nx_n+2\mu^{(2)}_nx_n/h}=e^{ix\cdot \tilde \xi_++2\mu^{(2)}_nx_n/h},
\end{align*}
where 
\[
\tilde\xi_\pm=\bigg(\xi',\pm \frac{2}{h}\sqrt{1-\frac{h^2|\xi|^2}{4}}\mu^{(1)}_n\bigg).
\]
We assume further that $\mu^{(2)}_n=0$ and $\mu^{(1)}_n\ne 0$. Then $ |\tilde\xi_\pm|\to \infty$ as $h\to 0$. We have
\begin{equation}
\label{eq_intergals_comp}
\begin{aligned}
&\int_{\Sigma\cap B} \zeta_1\cdot (A^{(1)}-A^{(2)})  e^{x\cdot\zeta_1/h}e^{(x',-x_n)\cdot\overline{\zeta_2}/h} e^{\Phi_1(x)}e^{\overline{\Phi_2(x',-x_n)}}dx\\
&=
\zeta_1\cdot \int_{\Sigma\cap B}  (A^{(1)}-A^{(2)}) e^{ix\cdot \tilde \xi_+}  e^{\Phi_1^{(0)}(x)+\overline{\Phi^{(0)}_2(x',-x_n)}}dx \\
&+ \zeta_1\cdot \int_{\Sigma\cap B}  (A^{(1)}-A^{(2)}) e^{ix\cdot \tilde \xi_+}  (e^{\Phi_1(x)+\overline{\Phi_2(x',-x_n)}}-
e^{\Phi_1^{(0)}(x)+\overline{\Phi^{(0)}_2(x',-x_n)}})dx
\to 0, 
\end{aligned}
\end{equation}
as $h\to 0$. Here the first integral in the right hand side of \eqref{eq_intergals_comp} goes to zero as $h\to 0$ by the Riemann-Lebesgue lemma, and the second one goes to zero, since $\|\Phi_j-\Phi_j^{(0)}\|_{L^\infty}\to 0$, as $h\to 0$. 
Therefore, multiplying \eqref{eq_identity_main} by $h$ and letting $h\to 0$, we get
\begin{align*}
\lim_{h\to 0}&\bigg(\int_{\Sigma\cap B} (A^{(1)}-A^{(2)})\cdot (-i\zeta_1+i\overline{\zeta_2})e^{ix\cdot \xi}e^{\Phi_1(x)+\overline{\Phi_2(x)}}dx\\
&+
\int_{\Sigma\cap B} (A^{(1)}-A^{(2)})\cdot (-i\zeta_1^*+i\overline{\zeta_2^*})e^{i(x',-x_n)\cdot \xi}e^{\Phi_1(x',-x_n)+\overline{\Phi_2(x',-x_n)}}dx
\bigg)=0.
\end{align*}
This implies that
\begin{align*}
(i\mu^{(1)}&+\mu^{(2)})\cdot \int_{\Sigma\cap B} (A^{(1)}-A^{(2)})e^{ix\cdot\xi}e^{\Phi_1^{(0)}(x)+\overline{\Phi_2^{(0)}(x)}}dx\\
&+(i(\mu^{(1)})'+(\mu^{(2)})', -(i\mu^{(1)}_n+\mu^{(2)}_n))\\
&\cdot \int_{\Sigma\cap B} (A^{(1)}-A^{(2)})e^{i(x',-x_n)\cdot \xi}e^{\Phi_1^{(0)}(x',-x_n)+\overline{\Phi_2^{(0)}(x',-x_n)}}dx
=0.
\end{align*}
Making a change of variables, we get
\[
(i\mu^{(1)}+\mu^{(2)})\cdot \int_{(\Sigma\cap  B)\cup (\Sigma\cap  B)_0^*} (\tilde A^{(1)}-\tilde A^{(2)})e^{ix\cdot\xi}e^{\Phi_1^{(0)}(x)+\overline{\Phi_2^{(0)}(x)}}dx=0, 
\]
where $\mu^{(2)}_n=0$ and $\mu^{(1)}_n\ne 0$.  
At this point, we can repeat the arguments, used in the proof of Proposition \ref{prop_32}, and conclude that
\begin{equation}
\label{eq_plane_1}
(i\mu^{(1)}+\mu^{(2)})\cdot \int_{(\Sigma\cap  B)\cup (\Sigma\cap  B)_0^*} (\tilde A^{(1)}-\tilde A^{(2)})e^{ix\cdot\xi}dx=0, 
\end{equation}
for all $\xi,\mu^{(1)},\mu^{(2)}\in\R^n$ such that 
\begin{equation}
\label{eq_vect_1}
\xi\cdot \mu^{(1)}=\xi\cdot\mu^{(2)}=\mu^{(1)}\cdot\mu^{(2)}=0,\  |\mu^{(1)}|=|\mu^{(2)}|=1,\ 
\mu^{(2)}_n=0,\ \mu^{(1)}_n\ne 0.
\end{equation}
Replacing the vector $\mu^{(1)}$ by $-\mu^{(1)}$, we obtain that 
\begin{equation}
\label{eq_plane_2}
(-i\mu^{(1)}+\mu^{(2)})\cdot \int_{(\Sigma\cap  B)\cup (\Sigma\cap  B)_0^*} (\tilde A^{(1)}-\tilde A^{(2)})e^{ix\cdot\xi}dx=0. 
\end{equation}
Hence, \eqref{eq_plane_1} and \eqref{eq_plane_2} imply that
\begin{equation}
\label{eq_plane_3}
\mu\cdot \int_{(\Sigma\cap  B)\cup (\Sigma\cap  B)_0^*} (\tilde A^{(1)}-\tilde A^{(2)})e^{ix\cdot\xi}dx=0,
\end{equation}
for all $\mu\in \textrm{span}\{\mu^{(1)},\mu^{(2)}\}$ and all $\xi\in \R^{n}$ such that  \eqref{eq_vect_1} holds. 

We need the following result. 

\begin{prop}
\label{prop_curl_thm_2}

We have
\[
\p_j(\tilde A_k^{(1)}- \tilde A_k^{(2)})-\p_k( \tilde A_j^{(1)}- \tilde A_j^{(2)})=0\quad\textrm{in }(\Sigma\cap  B)\cup (\Sigma\cap  B)_0^*,\quad 1\le j,k\le n.
\]
\end{prop}

\begin{proof}

Let first $n=3$. Then for any vector $\xi\in\R^3$ such that $\xi_{1}^2+\xi_{2}^2>0$, the vectors
\begin{align*}
\mu^{(2)}&=\bigg(\frac{-\xi_2}{\sqrt{\xi_{1}^2+\xi_{2}^2}}, \frac{\xi_1}{\sqrt{\xi_{1}^2+\xi_{2}^2}},0\bigg),\\
\tilde \mu^{(1)}&=(-\xi_1\xi_3,-\xi_2\xi_3,\xi_{1}^2+\xi_{2}^2),\quad \mu^{(1)}=\tilde \mu^{(1)}/|\tilde \mu^{(1)}|,
\end{align*}
satisfy \eqref{eq_vect_1}. 
Thus, for any vector $\xi\in\R^3$ such that $\xi_{1}^2+\xi_{2}^2>0$,
\eqref{eq_plane_3} says that
\begin{equation}
\label{eq_plane_4}
\mu \cdot  v(\xi)=0, \quad v(\xi):=\hat{\tilde A^{(1)}\chi}(\xi)-\hat{\tilde A^{(2)}\chi}(\xi),
\end{equation}
for all $\mu\in \textrm{span}\{\mu^{(1)},\mu^{(2)}\}$. Here $\chi$ is the characteristic function of the set $(\Sigma\cap  B)\cup (\Sigma\cap  B)_0^*$.  
For any vector $\xi\in \R^n$, we have the following decomposition,
\[
v(\xi)=v_{\xi}(\xi)+v_\perp(\xi),
\] 
where $\Re v_{\xi}(\xi)$, $\Im v_{\xi}(\xi)$ are multiples of $\xi$, and  $\Re v_\perp(\xi)$, $\Im v_\perp(\xi)$ are orthogonal to $\xi$.  Since $n=3$, we have $\Re v_\perp(\xi), \Im v_\perp(\xi)\in \textrm{span}\{\mu^{(1)},\mu^{(2)}\}$, and therefore, it follows from \eqref{eq_plane_4} that $v_\perp(\xi)=0$, 
for all $\xi\in \R^n$ such that  $\xi_{1}^2+\xi_{2}^2>0$. 
Hence, 
\begin{equation}
\label{eq_plane_5}
v(\xi)=\alpha(\xi)\xi. 
\end{equation}
Let $\mu(\xi,j,k):=-\xi_k e_j+\xi_j e_k$, $1\le j,k\le 3$, $j\ne k$. Here $e_j$ is the standard orthonormal basis in $\R^3$. 
Then \eqref{eq_plane_5} implies that
\[
\mu(\xi,j,k)\cdot  v(\xi)=0,
\]
for all $\xi\in \R^n$ such that  $\xi_{1}^2+\xi_{2}^2>0$, and therefore
\[
\xi_j\cdot (\hat{\tilde A_k^{(2)}\chi}(\xi)-\hat{\tilde A_k^{(1)}\chi}(\xi))-\xi_k\cdot (\hat{\tilde A_j^{(2)}\chi}(\xi)-\hat{\tilde A_j^{(1)}\chi}(\xi))=0,\  1\le j,k\le 3,\  j\ne k,
\]
for all $\xi\in \R^3$ such that  $\xi_{1}^2+\xi_{2}^2>0$, and thus, everywhere, by the analyticity of the Fourier transform. This completes the proof in the case $n=3$.

Let $n\ge 4$.  Then for any vector $\xi=(\xi_1,\dots,\xi_n)\in\R^n$, $\xi_l\ne 0$, $l=1,\dots,n$, the vectors 
\[
\mu^{(2)}(\xi,j,k)=-\xi_k e_j+\xi_j e_k,\ \mu^{(1)}(\xi,j,k)=(-\xi_j\xi_n)e_j+(-\xi_k\xi_n)e_k+(\xi_j^2+\xi_k^2)e_n, 
\]
$1\le j,k<n$, $ j\ne k$, satisfy 
$\xi\cdot\mu^{(1)}(\xi,j,k)=\xi\cdot \mu^{(2)}(\xi,j,k)=\mu^{(1)}(\xi,j,k)\cdot\mu^{(2)}(\xi,j,k)=0$,  $\mu^{(2)}_n(\xi,j,k)=0$, and $\mu^{(1)}_n(\xi,j,k)\ne 0$.
Thus, \eqref{eq_plane_3} implies that 
\begin{equation}
\label{eq_four_n4}
\xi_j\cdot (\hat{\tilde A_k^{(2)}\chi}(\xi)-\hat{\tilde A_k^{(1)}\chi}(\xi))-\xi_k\cdot (\hat{\tilde A_j^{(2)}\chi}(\xi)-\hat{\tilde A_j^{(1)}\chi}(\xi))=0,\  1\le j,k< n,\  j\ne k,
\end{equation}
for all $\xi\in\R^n$, $\xi_l\ne 0$, $l=1,2,\dots,n$.

  Let  $\xi=(\xi_1,\dots,\xi_n)\in\R^n$, $\xi_l\ne 0$, $l=1,\dots,n$, and let $1\le j<n$. Choose the indices $k$ and $l$ such that the set $\{j,n,k,l\}$ consists of four distinct elements.   
   Then the vectors 
\[
\mu^{(1)}(\xi,j,n)=-\xi_n e_j+\xi_j e_n,\ \mu^{(2)}(\xi,j,n)=-\xi_ke_l+\xi_le_k, 
\]
satisfy 
$\xi\cdot\mu^{(1)}(\xi,j,n)=\xi\cdot \mu^{(2)}(\xi,j,n)=\mu^{(1)}(\xi,j,n)\cdot\mu^{(2)}(\xi,j,n)=0$,  $\mu^{(2)}_n(\xi,j,n)=0$, and $\mu^{(1)}_n(\xi,j,n)\ne 0$.
Hence, it follows from 
\eqref{eq_plane_3}  that 
\begin{equation}
\label{eq_four_n4_2}
\xi_j\cdot (\hat{\tilde A_n^{(2)}\chi}(\xi)-\hat{\tilde A_n^{(1)}\chi}(\xi))-\xi_n\cdot (\hat{\tilde A_j^{(2)}\chi}(\xi)-\hat{\tilde A_j^{(1)}\chi}(\xi))=0,\  1\le j< n,
\end{equation}
for all $\xi\in\R^n$, $\xi_l\ne 0$, $l=1,2,\dots,n$.  In the case $n\ge 4$, the claim of the proposition follows from 
\eqref{eq_four_n4} and \eqref{eq_four_n4_2}.

\end{proof}

By Proposition \ref{prop_curl_thm_2}, we obtain that $d\tilde A^{(1)}=d \tilde A^{(2)}$ in $(\Sigma\cap  B)\cup (\Sigma\cap  B)_0^*$.  Arguing as in the proof of Theorem \ref{thm_main_1}, we see that  there exists $\Psi\in C^{1,1}(\overline{\Sigma\cup \Sigma_0^*})$ with compact support such that 
\[
\tilde A^{(1)}-\tilde A^{(2)}=\nabla \Psi\quad \textrm{in}\ (\Sigma\cap  B)\cup (\Sigma\cap  B)_0^*,
\]
and $\Psi=0$ along $\p ((\Sigma\cap  B)\cup (\Sigma\cap  B)_0^*)$.  Since in particular $\Psi=0$ on $\Gamma_1$, we have
\[
\mathcal{N}_{A^{(2)}+\nabla\Psi,q^{(2)}}(f)|_{\gamma_1'}=\mathcal{N}_{A^{(2)},q^{(2)}}(f)|_{\gamma_1'},
\] 
for any $f\in H^{3/2}(\Gamma_1)$, $\supp (f)\subset\gamma_1$. Thus,
\[
\mathcal{N}_{A^{(1)},q^{(1)}}(f)|_{\gamma_1'}=\mathcal{N}_{A^{(1)},q^{(2)}}(f)|_{\gamma_1'},
\]
for any $f\in H^{3/2}(\Gamma_1)$, $\supp (f)\subset\gamma_1$. Hence, we may and shall assume that $A^{(1)}=A^{(2)}$.

As for the electric potentials $q^{(1)}$, $q^{(2)}$, continuing to argue as in the proof of Theorem \ref{thm_main_1},  we arrive at 
\begin{equation}
\label{eq_Four_q}
\int_{(\Sigma\cap  B)\cup (\Sigma\cap  B)_0^*}(q^{(1)}-q^{(2)})e^{ix\cdot\xi}dx=0,
\end{equation}
for all $\xi\in \R^n$ such that there exist $\mu^{(1)},\mu^{(2)}\in\R^n$, satisfying 
\eqref{eq_vect_1}.  For any vector $\xi\in\R^n$ such that $\xi_{n-2}^2+\xi_{n-1}^2>0$, 
the vectors
\begin{align*}
\mu^{(2)}&=\bigg(0_{\R^{n-3}},\frac{-\xi_{n-1}}{\sqrt{\xi_{n-2}^2+\xi_{n-1}^2}}, \frac{\xi_{n-2}}{\sqrt{\xi_{n-2}^2+\xi_{n-1}^2}},0\bigg),\\
\tilde{\mu^{(1)}}&=\bigg(0_{\R^{n-3}},\frac{-\xi_n\xi_{n-2}}{\sqrt{\xi_{n-2}^2+\xi_{n-1}^2}}, \frac{-\xi_n\xi_{n-1}}{\sqrt{\xi_{n-2}^2+\xi_{n-1}^2}}, \sqrt{\xi_{n-2}^2+\xi_{n-1}^2}\bigg),\ \mu^{(1)}=\frac{\tilde{\mu^{(1)}}}{|\tilde{\mu^{(1)}}|},
\end{align*}
satisfy \eqref{eq_vect_1}. 
Thus, \eqref{eq_Four_q} holds for all $\xi\in\R^n$ such that $\xi_{n-2}^2+\xi_{n-1}^2>0$, and therefore, by the analyticity of the Fourier transform, for all $\xi\in \R^n$. This completes  the proof of Theorem \ref{thm_main_2}.

\textbf{Remark}. 
We shall finish this section by making a remark concerning inverse problems for the Schr\"odinger equation  in a slab, which arise in optical tomography \cite{O1}. 
In optical  diffusion tomography one   reconstructs the optical material parameters
inside an object by measuring the light transmitted and scattered through the object. 
There,
 a time harmonic diffusion equation is obtained by using an  approximation of the radiative transfer equation.
The so-called photon density  function has the form $\Phi(x,t)=\Re \,(e^{i\omega t}\Phi(x))$, where  $\Phi(x)$ satisfies
the equation
\beq
\label{Diff appr}
-\nabla\,\cdotp ( \kappa(x)\nabla \Phi(x))+\bigg(\mu_a(x)+\frac{i\omega}c\bigg) \Phi(x)=0,
\eeq
where $\kappa(x)=(3\mu_a(x)+3\mu'_s(x))^{-1}$ is the diffusion coefficient, $\mu_a(x)$
 is the absorption coefficient of the medium,
  $\mu_s'(x)$
 is the reduced scattering  coefficient of the medium,
$\omega$ is the frequency, and $c$ is the speed of light. Equation (\ref{Diff appr})
yields
\beq\label{Diff appr2}
-\Delta \Phi(x)-\kappa(x)^{-1}\nabla \kappa(x)\,\cdotp \nabla
\Phi(x)+\kappa(x)^{-1}\bigg(\mu_a(x)+\frac{i\omega}c\bigg) \Phi(x)=0,
\eeq
which is of the form (\ref{eq_Dirichlet_problem}). When the sources are on the upper
boundary hyperplane
of the slab, the function $\Phi$ satisfies on the lower boundary hyperplane
a Robin boundary condition $\Phi+2\mathcal{A}\kappa \p_{\nu}\Phi =0$, where the parameter
$\mathcal{A}(x)$ depends on the properties of the materials on both sides of the lower
boundary hyperplane, see \cite{O3,O4}.
For small values of $\mathcal{A}\kappa$, corresponding to the case when
 scattering  or absorption is high,  this boundary condition
can be approximated by the Dirichlet boundary condition.
For the study of inverse problems in optical tomography on bounded domains, see \cite{O1},
and references therein.  In particular, the first order terms in
(\ref{Diff appr2}) are important in explaining the non-uniqueness encountered in the imaging problems
 in optical tomography, see \cite{O2}.

\section{Remarks on inverse problems on bounded domains}

\subsection{Proof of Theorem \ref{thm_AU}}

Let $\Omega'\subset\subset\Omega$ be a bounded domain with $C^\infty$ smooth boundary such that $\Omega\setminus\overline{\Omega'}$ is connected and $\Omega'$ contains $\supp(A^{(1)}-A^{(2)})$ and $\supp(q^{(1)}-q^{(2)})$. 

Let $u_1\in H^2(\Omega)$ be the solution to the Dirichlet problem,
\begin{align*}
\mathcal{L}_{A^{(1)},q^{(1)}}u_1&=0\quad \textrm{in}\quad \Omega,\\
u_1&=f\quad\textrm{on}\quad \p \Omega,
\end{align*}
for  some $f\in H^{3/2}(\p \Omega)$ such that $\supp(f)\subset \gamma_1$. 
  Let also $v\in H^2(\Omega)$ be the solution of the following problem,
\begin{align*}
\mathcal{L}_{A^{(2)},q^{(2)}}v&=0\quad \textrm{in}\quad \Omega,\\
v&=f\quad\textrm{on}\quad \p \Omega. 
\end{align*}
 Setting $w=v-u_1\in H^1_0(\Omega)\cap H^2(\Omega)$, we get
\begin{equation}
\label{eq_w_bounded}
\begin{aligned}
\mathcal{L}_{A^{(2)},q^{(2)}} w&=(A^{(1)}-A^{(2)})\cdot D u_1+D\cdot ((A^{(1)}-A^{(2)})u_1)\\
&+((A^{(1)})^2-(A^{(2)})^2+q^{(1)}-q^{(2)})u_1\quad \textrm{in }\Omega.
\end{aligned}
\end{equation}
By our assumptions, 
\[
(\p_{\nu}+iA^{(1)}\cdot\nu) u_1|_{\gamma_2}=(\p_\nu+i A^{(2)}\cdot\nu) v|_{\gamma_2},
\]
and since $A^{(1)}=A^{(2)}$ in a neighborhood of $\p \Omega$, we have  $\p_{\nu}w=0$ on $\gamma_2$. 
It follows from \eqref{eq_w_bounded} that $w$ is a solution to
\[
\mathcal{L}_{A^{(2)},q^{(2)}} w=0 \quad \textrm{in}\quad \Omega\setminus{\overline{\Omega'}},
\]
and $w=\p_\nu w=0$ on $\gamma_2$.  As $A^{(2)}\in W^{1,\infty}(\Omega)$, $q\in L^\infty(\Omega)$, and $\Omega\setminus{\overline{\Omega'}}$ is connected, 
 by unique continuation,  we obtain that $w=0$ in $\Omega\setminus{\overline{\Omega'}}$, see \cite[Corollary 1.38]{Choulli_book}. 
Thus,  $w=\p_\nu w=0$ on $\p \Omega'$.  

Let $u_2\in H^2(\Omega')$ be a solution of the equation 
\begin{equation}
\label{eq_u_2_sec3}
\mathcal{L}_{\overline{A^{(2)}},\overline{q^{(2)}}}u_2=0\quad \textrm{in}\quad \Omega'.
\end{equation}
Then using Green's formula,   we have
\begin{align*}
(\mathcal{L}_{A^{(2)},q^{(2)}}w,u_2)_{L^2(\Omega')}&=
(w,(\p_\nu+i\nu\cdot\overline{A^{(2)}})u_2)_{L^2(\p \Omega')}-
((\p_{\nu}+i\nu \cdot A^{(2)})w,u_2)_{L^2(\p \Omega')}\\
&=0.
\end{align*}
This together with \eqref{eq_w_bounded} implies  that
\begin{align*}
\int_{\Omega'} (A^{(1)}-A^{(2)})\cdot ((Du_1)\overline{u_2}+u_1\overline{Du_2})dx-i \int_{\p \Omega'}(A^{(1)}-A^{(2)})\cdot\nu u_1\overline{u_2}dS\\
+\int_{\Omega'} ((A^{(1)})^2-(A^{(2)})^2+q^{(1)}-q^{(2)})u_1\overline{u_2}dx=
0.
\end{align*}
Therefore, since $A^{(1)}=A^{(2)}$ on $\p \Omega'$, we get 
\begin{equation}
\label{eq_identity_main_sec3}
\begin{aligned}
\int_{\Omega'} (A^{(1)}-A^{(2)})\cdot ((Du_1)\overline{u_2}+u_1\overline{Du_2})dx\\
+\int_{\Omega'} ((A^{(1)})^2-(A^{(2)})^2+q^{(1)}-q^{(2)})u_1\overline{u_2}dx=0,
\end{aligned}
\end{equation}
for any $u_2\in H^2(\Omega')$ satisfying \eqref{eq_u_2_sec3} and any $u_1\in W(\Omega)$, where 
\[
W(\Omega)=\{u_1\in H^2(\Omega): \mathcal{L}_{A^{(1)},q^{(1)}}u_1=
0\ \textrm{in }\Omega,\ \supp(u_1|_{\p \Omega})\subset \gamma_1
\}. 
\]
Let 
\[
\tilde W(\Omega')=\{u_1\in H^2(\Omega'):\ \mathcal{L}_{A^{(1)},q^{(1)}}u_1=0\ \textrm{in }\Omega'\}. 
\]

We need the following Runge type approximation result. 

\begin{prop}
\label{prop_Runge_sec3}
The space $W(\Omega)$ is dense in $\tilde W(\Omega')$ in $L^2(\Omega')$-topology.
\end{prop}

\begin{proof}
By the Hahn-Banach theorem, we need to show that  for any $g\in L^2(\Omega')$ such that
\[
\int_{\Omega'} g\overline udx=0\quad \textrm{for any } u\in W(\Omega),
\]
we have
\[
\int_{\Omega'} g\overline{v}dx=0\quad \textrm{for any } v\in \tilde W(\Omega'). 
\]
Continue $g$ by zero to $\Omega\setminus\Omega'$ and consider the Dirichlet problem,
\begin{equation}
\label{eq_U_sec3}
\begin{aligned}
\mathcal{L}_{\overline{A^{(1)}},\overline{q^{(1)}}}U&=g\quad \textrm{in}\quad \Omega,\\
U&=0\quad\textrm{on}\quad \p\Omega. 
\end{aligned}
\end{equation} 
As the assumption (A) holds for the operator $\mathcal{L}_{A^{(1)},q^{(1)}}$,  it also holds for the adjoint operator $\mathcal{L}_{\overline{A^{(1)}},\overline{q^{(1)}}}$, and therefore, the problem \eqref{eq_U_sec3} has a unique solution $U\in H^2(\Omega)\cap H^1_0(\Omega)$.  
For any $u\in W(\Omega)$, using the Green formula, we have
\begin{align*}
0&=\int_{\Omega} g\overline udx=\int_{\Omega}(\mathcal{L}_{\overline{A^{(1)}},\overline{q^{(1)}}}U)\overline{u}dx=\int_{\Omega} U\overline{\mathcal{L}_{A^{(1)},q^{(1)}} u}dx\\
&+ \int_{\p \Omega }U\overline{(\p_\nu +i\nu\cdot A^{(1)})u}dS-\int_{\p \Omega}(\p_\nu +i \nu\cdot\overline{A^{(1)}})U\overline{u}dS=-\int_{\p \Omega}\p_\nu U\overline{u}dS.
\end{align*}
Since $u|_{\p \Omega}$ can be an arbitrary smooth function, supported in  $\gamma_1$, we conclude that $\p_\nu U|_{\gamma_1}=0$. 
Hence, $U$ satisfies the equation 
$\mathcal{L}_{\overline{A^{(1)}},\overline{q^{(1)}}} U=0$ in $\Omega\setminus \overline{\Omega'}$, and  $U=\p_\nu U=0$ on $\gamma_1$. 
Thus, by  unique continuation, $U=0$ in $\Omega\setminus \overline{\Omega'}$, and therefore, we have $U=\p_\nu U=0$ on $\p \Omega'$.

For any $v\in \tilde W(\Omega')$, using the Green formula, we get
\begin{align*}
\int_{\Omega'} g\overline{v}dx=\int_{\Omega'}(\mathcal{L}_{\overline{A^{(1)}},\overline{q^{(1)}}}U)\overline{v}dx
=\int_{\Omega'} U\overline{\mathcal{L}_{A^{(1)},q^{(1)}} v}dx\\
+ \int_{\p \Omega'}U\overline{(\p_\nu +i\nu\cdot A^{(1)})v}dS-\int_{\p \Omega'}(\p_\nu +i \nu\cdot\overline{A^{(1)}})U\overline{v}dS=0.
\end{align*}
The proof is complete. 
\end{proof}

Since $A^{(1)}=A^{(2)}$ on $\p \Omega'$,  in the same way as in the proof of Theorem \ref{thm_main_1}, we conclude that an application of  Proposition \ref{prop_Runge_sec3} implies that
\eqref{eq_identity_main_sec3} is valid for any $u_1\in  \tilde W(\Omega')$ and $u_2\in H^2(\Omega')$ satisfying \eqref{eq_u_2_sec3} .

Let $B\subset\R^n$ be an open ball such that $\Omega'\subset\subset B$. Since $A^{(1)}=A^{(2)}$ and $q^{(1)}=q^{(2)}$ on $\p \Omega'$, we can extend $A^{(j)}$ and $q^{(j)}$ to $B$ so that the extensions, which we shall denote with by same letters, agree on $B\setminus\Omega'$, have compact support, and satisfy $A^{(j)}\in W^{1,\infty}(B)$, $q^{(j)}\in L^\infty(B)$.  
Hence, it follows from \eqref{eq_identity_main_sec3} that 
\begin{equation}
\label{eq_identity_main_sec3_1}
\begin{aligned}
\int_{B} (A^{(1)}-A^{(2)})\cdot ((Du_1)\overline{u_2}+u_1\overline{Du_2})dx\\
+\int_{B} ((A^{(1)})^2-(A^{(2)})^2+q^{(1)}-q^{(2)})u_1\overline{u_2}dx=0,
\end{aligned}
\end{equation}
for any $u_1,u_2\in H^2(B)$, which solve 
\[
 \mathcal{L}_{A^{(1)},q^{(1)}}u_1=0\ \textrm{in } B,\quad  \mathcal{L}_{\overline{A^{(2)}},\overline{q^{(2)}}}u_2=0\ \textrm{in } B. 
\]
 By Proposition \ref{prop_CGO_Lip} and Remark \ref{rem_com_geom_1}, we can construct complex geometric optics solutions $u_1$ and $u_2$ on $B$,  with $\zeta_1$ and $\zeta_2$ given by 
\eqref{eq_zeta_1_2}.  Substituting the constructed complex geometric optics solutions into \eqref{eq_identity_main_sec3_1}, and proceeding similarly to  \cite{NakSunUlm_1995, Salo_diss, Sun_1993},  we complete the proof. See also the proof of Theorem \ref{thm_main_1}.

\subsection{Proof of Theorem \ref{thm_Isak}}

Notice first that without loss of generality, as in the proofs of Theorem \ref{thm_main_1} and Theorem \ref{thm_main_2},
we assume, as we may, that
\[
A^{(1)}\cdot \nu|_{\p \Omega}=A^{(2)}\cdot \nu|_{\p \Omega}=0. 
\]
Then in the standard way as above, we obtain the following integral identity,
\begin{equation}
\label{eq_iden_isak}
\begin{aligned}
\int_{\Omega} (A^{(1)}-A^{(2)})\cdot ((Du_1)\overline{u_2}+u_1\overline{Du_2})dx\\
+\int_{\Omega} ((A^{(1)})^2-(A^{(2)})^2+q^{(1)}-q^{(2)})u_1\overline{u_2}dx=0,
\end{aligned}
\end{equation}
valid for all $u_1,u_2\in H^2(\Omega)$ such that 
\begin{equation}
\label{eq_u_1_isak}
\mathcal{L}_{A^{(1)},q^{(1)}}u_1=0\quad  \textrm{in } \Omega,\quad 
u_1|_{x_n=0}=0,
\end{equation}
\begin{equation}
\label{eq_u_2_isak}
 \mathcal{L}_{\overline{A^{(2)}},\overline{q^{(2)}}}u_2=0\quad  \textrm{in } \Omega,\quad 
u_2|_{x_n=0}=0. 
\end{equation}
Using the method of reflection as in Theorem \ref{thm_main_2}, we construct  complex geometric optics solutions $u_1$ and $u_2$, as given by \eqref{eq_u_1-CGO_1} and \eqref{eq_u_2-CGO_2_l_2}, and satisfying \eqref{eq_u_1_isak} and \eqref{eq_u_2_isak}, respectively.  

Substituting the complex geometric optics solutions $u_1$ and $u_2$ into \eqref{eq_iden_isak}, similarly to the proof of TheoremÊ\ref{thm_main_2}, we obtain that 
\begin{equation}
\label{eq_id_main_isak}
(i\mu^{(1)}+\mu^{(2)})\cdot \int_{\Omega\cup \Omega_0^*} (\tilde A^{(1)}-\tilde A^{(2)})e^{ix\cdot\xi}dx=0, 
\end{equation}
for all $\xi,\mu^{(1)},\mu^{(2)}\in\R^n$ such that 
\[
\xi\cdot \mu^{(1)}=\xi\cdot\mu^{(2)}=\mu^{(1)}\cdot\mu^{(2)}=0,\  |\mu^{(1)}|=|\mu^{(2)}|=1,\ 
\mu^{(2)}_n=0,\ \mu^{(1)}_n\ne 0.
\]
Here $\Omega_0^*=\{(x',x_n)\in \R^n :(x',-x_n)\in \Omega\}$.

At this point it is convenient to apply the boundary reconstruction results of \cite{Brown_Salo} to conclude that 
$A^{(1)}=A^{(2)}$ along $\overline{\gamma}$. Thus, it follows that $\tilde A^{(1)}=\tilde A^{(2)}$ along $\p (\Omega\cup \Omega_0^*)$.  Therefore, we may extend $\tilde A^{(j)}$, $j=1,2$, to compactly supported $W^{1,\infty}$ vector fields on some large ball $B\subset \R^n$, such that $\Omega\cup \Omega_0^*\subset\subset B$, in such a way that $\tilde A^{(1)}=\tilde A^{(2)}$ in $B\setminus(\Omega\cup \Omega_0^*)$.  
Hence, \eqref{eq_id_main_isak} is replaced by
\[
(i\mu^{(1)}+\mu^{(2)})\cdot \int_{B} (\tilde A^{(1)}-\tilde A^{(2)})e^{ix\cdot\xi}dx=0. 
\]
By Proposition \ref{prop_curl_thm_2}, we get $d\tilde A^{(1)}=d\tilde A^{(2)}$ in $B$, and therefore, there exists $\Psi\in C^{1,1}(\overline B)$
such that 
\[
\tilde A^{(1)}-\tilde A^{(2)}=\nabla \Psi\quad\textrm{in}\quad B. 
\]
It follows that $\nabla \Psi=0$ in  $B\setminus(\Omega\cup \Omega_0^*)$, and thus, $\Psi$ is constant along the connected set $\p (\Omega\cup \Omega_0^*)$. In particular,  $\Psi$ is constant along  $\overline{\gamma}$, and modifying $\Psi$ by constant, we may assume that $\Psi=0$ along $\overline{\gamma}$. 
Hence, we may and shall assume that $A^{(1)}=A^{(2)}$ in $\Omega$.  When recovering the electric potentials $q^{(1)}$ and $q^{(2)}$, 
we argue as in the end of the proof of Theorem \ref{thm_main_2}. This completes the proof.

\begin{appendix}

\section{Solvability of the direct problem in an infinite slab}

The purpose of this appendix is to provide a self-contained discussion of the solvability of the Dirichlet problem 
\eqref{eq_Dirichlet_problem}
for the magnetic Schr\"odinger operator in an infinite slab.  
Let 
\[
\Sigma=\{x=(x',x_n)\in\R^n:x'=(x_1,\dots,x_{n-1})\in \R^{n-1},0<x_n<L\}\subset\R^n, 
\]
$n\ge 3$, $L>0$,
be an infinite slab between two parallel hyperplanes 
\[
\Gamma_1=\{x\in\R^n:x_n=L\}\quad\textrm{and}\quad \Gamma_2=\{x\in\R^n:x_n=0\}. 
\]

By the Poincar\'e inequality in an infinite slab $\Sigma$, see \cite[Theorem 4.29]{Grubbbook2009},
 the quadratic form
\[
u\mapsto \int_\Sigma |\nabla u|^2dx
\]
is non-negative densely defined closed on $H^1_0(\Sigma)$. Associated with this quadratic form, the Laplace operator $-\Delta$, equipped with the domain
\[
\mathcal{D}(-\Delta)=\{u\in H^1_0(\Sigma):\Delta u\in L^2(\Sigma)\}, 
\]
is a non-negative self-adjoint operator on $L^2(\Sigma)$.

\begin{prop}

\label{prop_1}

We have  $\mathcal{D}(-\Delta)=H^1_0(\Sigma)\cap H^2(\Sigma)$. Furthermore, the spectrum of $-\Delta$ is purely absolutely continuous and is equal to $[\pi^2/L^2,+\infty)$. 
\end{prop}

\begin{proof}
For $F\in L^2(\Sigma)$, we consider 
\[
-\Delta u=F,\quad u\in \mathcal{D}(-\Delta). 
\]
Taking the Fourier decompositions
with respect to the variable $x_n\in [0,L]$,  
\begin{equation}
\label{eq_Fourier_dec}
\begin{aligned}
u(x',x_n)&=\sum_{l=1}^\infty u_l(x')\sin\frac{l\pi x_n}{L}, \quad x'\in \R^{n-1},\quad x_n\in [0,L],\\
F(x',x_n)&=\sum_{l=1}^\infty F_l(x')\sin\frac{l\pi x_n}{L},
\end{aligned}
\end{equation}
we have 
\begin{equation}
\label{eq_separ_var}
\bigg(-\Delta_{x'}+\frac{l^2\pi^2}{L^2}\bigg)u_l(x')=F_l(x'), \quad x'\in \R^{n-1}, \quad l=1,2,\dots.
\end{equation}
Here the Fourier coefficients $u_l$ of $u$ and $F_l$ of $F$ are given by
\begin{equation}
\label{eq_Four_coefficinets}
u_l(x')=\frac{2}{L}\int_0^L u(x)\sin \frac{l\pi x_n}{L}dx_n,
\end{equation}
\[
F_l(x')=\frac{2}{L}\int_0^L F(x)\sin \frac{l\pi x_n}{L}dx_n. 
\]
The functions $F_l\in L^2(\R^{n-1})$ and we have the 
Parseval identity 
\[
\|F\|^2_{L^2(\Sigma)}=\frac{L}{2}\sum_{l=1}^\infty \|F_l\|^2_{L^2(\R^{n-1})}.
\]
The operator 
\[
-\Delta_{x'}+\frac{l^2\pi^2}{L^2}, \quad l\ge 1, 
\]
 on $\R^{n-1}$, equipped with the domain $H^2(\R^{n-1})$, is self-adjoint on $L^2(\R^{n-1})$ with purely absolutely continuous spectrum $[l^2\pi^2/L^2,+\infty)$.  Hence,  \eqref{eq_separ_var} has the unique solution 
 \[
 u_l(x')=\bigg(-\Delta_{x'}+\frac{l^2\pi^2}{L^2}\bigg)^{-1}F_l(x')\in H^2(\R^{n-1}),
 \]
 and moreover, 
 \begin{equation}
 \label{eq_u_l_l_2}
 \|u_l\|_{L^2(\R^{n-1})}\le \frac{L^2}{l^2\pi^2}\|F_l\|_{L^2(\R^{n-1})},
 \end{equation}
 \begin{equation}
 \label{eq_u_l_h_2}
  \|u_l\|_{H^2(\R^{n-1})}\le C(\|u_l\|_{L^2(\R^{n-1})}+\|\Delta_{x'} u_l\|_{L^2(\R^{n-1})})\le C\|F_l\|_{L^2(\R^{n-1})},
 \end{equation}
where $C$ is independent of  $l$.  By interpolation,
 \begin{equation}
 \label{eq_u_l_h_1}
\|u_l\|_{H^1(\R^{n-1})}\le \frac{C}{l}\|F_l\|_{L^2(\R^{n-1})},
\end{equation}
where $C$ is independent of $l$.
By Parseval's identity and  \eqref{eq_u_l_l_2},  we have
\begin{align*}
\|u\|_{L^2(\Sigma)}^2&=\frac{L}{2}\sum_{l=1}^\infty\|u_l\|_{L^2(\R^{n-1})}^2\le C\sum_{l=1}^\infty \frac{1}{l^4}\|F_l\|_{L^2(\R^{n-1})}^2\le C\|F\|^2_{L^2(\Sigma)},\\
\|\p_{x_n}u\|_{L^2(\Sigma)}^2&=\|\sum_{l=1}^\infty \frac{l\pi}{L} u_l(x')\cos\frac{l\pi x_n}{L}\|_{L^2(\Sigma)}^2\\
&=\frac{L}{2} \sum_{l=1}^\infty \frac{l^2\pi^2}{L^2}\|u_l\|^2_{L^2(\R^{n-1})}\le C\|F\|^2_{L^2(\Sigma)},
\end{align*}
Using \eqref{eq_u_l_h_1}, we get
\[
\|\p_{x_j}u\|^2_{L^2(\Sigma)}=\frac{L}{2}\sum_{l=1}^\infty\|\p_{x_j} u_l\|^2_{L^2(\R^{n-1})}\le C\sum_{l=1}^\infty \frac{1}{l^2}\|F_l\|_{L^2(\R^{n-1})}^2\le C\|F\|^2_{L^2(\Sigma)},
\]
$j=1,2,\dots,n-1$. It follows from \eqref{eq_u_l_h_2} that
\[
\|\p^2_{x_j,x_k}u\|^2_{L^2(\Sigma)}=\frac{L}{2}\sum_{l=1}^\infty\|\p^2_{x_j,x_k} u_l\|^2_{L^2(\R^{n-1})}\le C\|F\|^2_{L^2(\Sigma)},
\]
$j,k=1,2,\dots,n-1$. Furthermore,
\[
\|\p^2_{x_n}u\|^2_{L^2(\Sigma)}=\frac{L}{2}\sum_{l=1}^\infty \frac{l^4\pi^4}{L^4}\|u_l\|^2_{L^2(\R^{n-1})}\le C\|F\|^2_{L^2(\Sigma)},
\]
\[
\|\p^2_{x_j,x_n}u\|^2_{L^2(\Sigma)}=\frac{L}{2}\sum_{l=1}^\infty \frac{l^2\pi^2}{L^2} \|\p_{x_j} u_l\|^2_{L^2(\R^{n-1})}\le C\|F\|^2_{L^2(\Sigma)},
\]
$j=1,2,\dots,n-1$. Hence, $u\in H^2(\Sigma)$. The proof is complete, since the statement concerning the spectrum of $-\Delta$ follows from the fact that  
\[
-\Delta=\bigoplus_{l=1}^\infty \bigg(-\Delta_{x'}+\frac{l^2\pi^2}{L^2}\bigg). 
\]

\end{proof}

\begin{prop}

\label{prop_essential_spec}
Let $A\in W^{1,\infty}(\Sigma,\C^n)\cap \mathcal{E}'(\overline{\Sigma}, \C^n)$ and $q\in L^\infty(\Sigma,\C)\cap \mathcal{E}'(\overline{\Sigma},\C)$.  
Then the operator $\mathcal{L}_{A,q}(x,D)$, equipped with the domain $H^1_0(\Sigma)\cap H^2(\Sigma)$ is closed and its essential spectrum is equal to $[\pi^2/L^2,+\infty)$. 

\end{prop}

\begin{proof}
We write 
\[
\mathcal{L}_{A,q}=-\Delta +2A\cdot D +\tilde q,\quad \tilde q=-i(\nabla\cdot A) +A^2+q\in  L^\infty(\Sigma,\C)\cap \mathcal{E}'(\overline{\Sigma},\C). 
\] 
Let $\chi\in C^\infty(\overline{\Sigma})$ be compactly supported and  $\chi=1$ near $\supp(\tilde q)\cup\supp(A)$.  Then the operator 
\[
\tilde q\Delta^{-1}:L^2(\Sigma)\to L^2(\Sigma)
\]
is compact, as a composition of the compact operator 
\[
\chi\Delta^{-1}:L^2(\Sigma)\to H^2(\Sigma)\cap \mathcal{E}'(\supp(\chi))\hookrightarrow L^2(\Sigma),
\]
and the bounded operator $\tilde q:L^2(\Sigma)\to L^2(\Sigma)$. 

For $j=1,\dots, n$,  the operator 
\[
A_jD_j\Delta^{-1}:L^2(\Sigma)\to L^2(\Sigma)
\]
is compact, as a composition of the compact operator 
\[
\chi D_j\Delta^{-1}:L^2(\Sigma)\to H^1(\Sigma)\cap \mathcal{E}'(\supp(\chi))\hookrightarrow L^2(\Sigma),
\]
and the bounded operator $A:L^2(\Sigma)\to L^2(\Sigma)$. Since relatively compact perturbations do not change the essential spectrum, the result follows in view of Proposition \ref{prop_1}.  

\end{proof}

Let $A\in W^{1,\infty}(\Sigma,\C^n)\cap \mathcal{E}'(\overline{\Sigma}, \C^n)$ and $q\in L^\infty(\Sigma,\C)\cap \mathcal{E}'(\overline{\Sigma},\C)$. 
Consider the following Dirichlet problem,
\begin{equation}
\label{eq_Dirichlet_app_first}
\begin{aligned}
(\mathcal{L}_{A,q}(x,D)-k^2)u&=F\quad \textrm{in}\quad \Sigma,\\
u|_{\p \Sigma}&=0,
\end{aligned}
\end{equation}
for some $k\ge 0$.

\textbf{(I)  The case $0\le k<\pi/L$. } 
We have the following immediate consequence of Proposition \ref{prop_essential_spec}.

\begin{cor} 
\label{cor_disc_spec_sol}
Assume that $0\le k<\pi/L$ and $k^2$ does not belong to the discrete spectrum of the operator $\mathcal{L}_{A,q}$, equipped with the domain $H^1_0(\Sigma)\cap H^2(\Sigma)$.  Then for any $F\in L^2(\Sigma)$, the problem \eqref{eq_Dirichlet_app_first} has a unique solution $u\in H^2(\Sigma)$. 

\end{cor}

\textbf{(II) The case $ k\ge \pi/L$. }  Our goal here is to study the solvability of the problem \eqref{eq_Dirichlet_app_first} for $F\in L^2(\Sigma)\cap\mathcal{E}'(\overline{\Sigma})$.   
In order to do this, let us first focus on the Dirichlet  problem for the Laplacian in the slab $\Sigma$,
\begin{equation}
\label{eq_Dirichlet_app}
\begin{aligned}
(-\Delta -k^2) u&=F\quad \textrm{in}\quad \Sigma,\\
u|_{\p \Sigma}&=0,
\end{aligned}
\end{equation}
for some $k\ge \pi/L$.  Taking the Fourier decomposition \eqref{eq_Fourier_dec}, we have
\begin{equation}
\label{eq_with_k}
\bigg(-\Delta_{x'}+\frac{l^2\pi^2}{L^2}-k^2\bigg)u_l(x')=F_l(x'), \quad x'\in \R^{n-1}, \quad l=1,2,\dots.
\end{equation}

\textbf{(II.i)} In the case when $l\in \N$ is such that $k>\pi l/L$, the equation
\eqref{eq_with_k} has a unique solution $u_l(x')$ satisfying the Sommerfeld radiation condition 
 \begin{equation}
 \label{eq_radiation_Som}
 u_l(x')=\mathcal{O}(|x'|^{-(n-2)/2}),\quad \bigg(\frac{\p}{\p |x'|}-ik_l\bigg)u_l(x')=o(|x'|^{-(n-2)/2}), 
  \end{equation}
 as $|x'|\to\infty$, see \cite{Col_Kress_book}. Here $k_l=\sqrt{k^2-l^2\pi^2/L^2}>0$. Notice that by elliptic regularity, $u_l\in H^2_{\textrm{loc}}(\R^{n-1})$. 
 
\textbf{(II.ii)} In the case when $l\in \N$ is such that $k<\pi l/L$, the equation
\eqref{eq_with_k} has a unique solution $u_l\in H^2(\R^{n-1})$.  

\textbf{(II.iii)} In the case when $l\in \N$ is such that $k=\pi l/L$, the equation 
\eqref{eq_with_k} has the following form,
\begin{equation}
\label{eq_with_k_2}
-\Delta_{x'}u_l(x')=F_l(x'), \quad x'\in \R^{n-1}. 
\end{equation}
In the case $n\ge 4$, \eqref{eq_with_k_2} has a unique solution $u_l\in H^2_{\textrm{loc}}(\R^{n-1})$ satisfying
\begin{equation}
\label{eq_radiation_Som_2}
u_l(x')=\mathcal{O}(|x'|^{3-n}),\quad \nabla_{x'} u_l(x')=\mathcal{O}(|x'|^{2-n}),
\end{equation}
as $|x'|\to \infty$.  Indeed, we have $u_l=E*F_l$, where $E(x')=C_{n}|x'|^{3-n}$ is the standard fundamental solution of $-\Delta$ in $\R^{n-1}$, $C_n\ne 0$ is a constant.

In the case $n=3$, we shall make the following assumption.

\begin{itemize}

\item[\textbf{(A.I)}]  In the case $n=3$, assume that $k$ is such that  $k\ne\pi l/L$, for all $l\in \N$.

\end{itemize}

The assumption (A.I) is motivated by the fact that  \eqref{eq_with_k_2} in general lacks solutions that are bounded on $\R^2$. 
Indeed, the general solution of \eqref{eq_with_k_2} in $\mathcal{S}'(\R^2)$ has the form,
\[
u_l=E_2*F_l+ p,
\]
where $E_2(x')=(2\pi)^{-1}\log|x'|$ is the standard fundamental solution of $-\Delta$ in $\R^2$, and $p$ is a harmonic polynomial. 

In what follows we shall need the  notation, 
\begin{align*}
\Sigma_{<R}:&=\Sigma\cap \{x\in \R^n:|x'|<R\},\\
\Sigma_{>R}:&=\Sigma\cap \{x\in \R^n:|x'|>R\}, \quad R>0, 
\end{align*}
and the following definition, which is closely related to the discussion in    \cite{Morgen_Werner_1987}.

\begin{defn}
\label{def_adm}
Assume that $u$ satisfies the following Dirichlet problem,
\begin{align*}
(-\Delta-k^2)u&=0\quad \textrm{in}\quad \Sigma_{>R},\\
u|_{\p\Sigma\cap\overline{\Sigma_{>R}}}&=0,
\end{align*}
for $R$ sufficiently large and $k\ge \pi/L$ such that  the assumption \emph{(A.I)} holds.    Let us write
\[
u(x)=\sum_{l=1}^\infty u_l(x')\sin\frac{l\pi x_n}{L}, 
\]
where the Fourier coefficients $u_l(x')$ are given by \eqref{eq_Four_coefficinets}. The function $u$ is said to be admissible, provided that the following conditions hold:
\begin{itemize}
\item[(i)] if $l<kL/\pi$, then  $u_l(x')$ satisfies the Sommerfeld radiation condition \eqref{eq_radiation_Som};
\item[(ii)] if $l>kL/\pi$, then $u_l\in H^2(\R^{n-1})$;

\item[(iii)] if $l=kL/\pi$ and $n\ge 4$, then $u_{kL/\pi}(x')$ satisfying \eqref{eq_radiation_Som_2}.

\end{itemize}

\end{defn}

Notice that if the function $u$ is admissible then $u\in H^2_{\textrm{loc}}(\overline{\Sigma})$. 
We obtain the following result. 

\begin{prop} Let $k\ge \pi/L$ and let  the assumption \emph{(A.I)} be satisfied. Then  for any 
$F\in L^2(\Sigma)\cap\mathcal{E}'(\overline{\Sigma})$, the  Dirichlet problem \eqref{eq_Dirichlet_app} for the Laplacian in the slab  has a unique admissible solution in the sense of Definition \emph{\ref{def_adm}}.
\end{prop}

Let us introduce the solution operator for the  Dirichlet problem \eqref{eq_Dirichlet_app},
\begin{align*}
R_0(k): L^2(\Sigma)\cap \mathcal{E}'(\overline{\Sigma})\to H^2_{\textrm{loc}}(\overline{\Sigma}),\ R_0(F)=u,
\end{align*}
where $u$ is the admissible solution of \eqref{eq_Dirichlet_app}.

In order to study the  solvability of the problem \eqref{eq_Dirichlet_app_first} for $F\in L^2(\Sigma)\cap\mathcal{E}'(\overline{\Sigma})$, in 
the case when $ k\ge \pi/L$, we shall use the Lax-Phillips method, see \cite{Isakov_book, Lax_Phillips,  LiUhl2010},  and to that end, we shall need the following assumption, which was also made in \cite{LiUhl2010}.

\begin{itemize}

\item[\textbf{(A.II)}] Let $k\ge \pi/L$ and let the assumption (A.I) be satisfied.   If $u$ is an admissible solution of the problem \eqref{eq_Dirichlet_app_first}  with $F=0$, then $u$ vanishes identically.  

\end{itemize}

The following result shows that in the self-adjoint case, assumption (A.II) is satisfied away from the embedded eigenvalues and the set of thresholds $\{(\pi l/L)^2:l=1,2,\dots\}$ of the operator $\mathcal{L}_{A,q}$. 

\begin{prop} 

\label{prop_admiss_fr}

Let $A\in W^{1,\infty}(\Sigma,\R^n)\cap \mathcal{E}'(\overline{\Sigma},\R^n)$, $q\in L^\infty(\Sigma,\C)\cap \mathcal{E}'(\overline{\Sigma},\C)$, and $\emph{\Im} q\le 0$. Assume that  $k\ge \pi/L$ is such that $k^2$ is not an eigenvalue of the operator $\mathcal{L}_{A,q}$, equipped with the domain $H^1_0(\Sigma)\cap H^2(\Sigma)$, and $k\ne \pi l/L$, for all $l=1,2,\dots$. 
Then the assumption \emph{(A.II)} is satisfied.  
\end{prop}

\begin{proof}
Let $u$ be an admissible solution of the  problem \eqref{eq_Dirichlet_app_first}  with $F=0$, and let $R>0$ be large so that $\supp(A)\subset\overline{\Sigma_{<R}}$. 
Multiplying \eqref{eq_Dirichlet_app_first}  by $\overline{u}$ and integrating over $\Sigma_{<R}$, using the fact that $A_j$ are real-valued, we get 
\begin{equation}
\label{eq_prop_as_2}
\begin{aligned}
0&=\int_{\Sigma_{<R}}(\mathcal{L}_{A,q}-k^2)u\overline{u}dx\\
&=\sum_{j=1}^n\bigg( \int_{\Sigma_{<R}}|(D_j+A_j)u|^2dx - i\int_{|x'|=R,0<x_n<L} \nu_j (D_j u)\overline{u}dx_ndS(x')\bigg)\\
&+ \int_{\Sigma_{<R}}(q-k^2)|u|^2dx.
\end{aligned}
\end{equation}
Taking the imaginary part in  \eqref{eq_prop_as_2}, we obtain that
\begin{equation}
\label{eq_im_part_1}
\Im \int_{|x'|=R,0<x_n<L} (\nu\cdot \nabla u)\overline{u}dx_ndS(x')=\int_{\Sigma_{<R}}\Im q |u|^2dx\le 0. 
\end{equation}
Let us write $u=u_0+u_1$, where
 \[
 u_0(x)=\sum_{l>kL/\pi}u_l(x')\sin\frac{l\pi x_n}{L},\quad  u_1(x)=\sum_{1\le l< kL/\pi}u_l(x')\sin\frac{l\pi x_n}{L},
 \]
 According to \cite{Morgen_Werner_1987}, we know that
 \begin{equation}
 \label{eq_resolv_p}
 u_0=\mathcal{O}(|x'|^{-n}),\quad \nabla u_0=\mathcal{O}(|x'|^{-n}),
 \end{equation}
 as $|x'|\to \infty$.  We have
 \[
 \int_{|x'|=R,0<x_n<L} (\nu\cdot \nabla u)\overline{u}dx_ndS(x')=I_0+I_1,
 \]
where by \eqref{eq_resolv_p},
\[
I_0=\int_{|x'|=R,0<x_n<L} (\nu \cdot\nabla u_0)\overline{u_0}dx_ndS(x')=\mathcal{O}(R^{-n-2}),
\]
as $R\to \infty$. 
Also using  the fact that $\nu\cdot\nabla=\frac{x'}{R}\cdot\nabla_{x'}$ along $|x'|=R$ , 
\[
I_1=\frac{L}{2}\sum_{1\le l< kL/\pi} \int_{|x'|=R} (\frac{x'}{R}\cdot\nabla_{x'} u_l(x'))\overline{u_l(x')}dS(x'),
\]
where 
$u_l$, $1\le l< kL/\pi$, satisfies the equation
\[
(-\Delta_{x'}+\frac{l^2\pi^2}{L^2}-k^2)u_l(x')=0,\quad x'\in \R^{n-1},\quad |x'|>R,
\]
and the Sommerfeld radiation condition \eqref{eq_radiation_Som}.  Then $u_l$ has the following asymptotic behavior 
\[
u_l(x')=a_l(\theta) \frac{e^{ik_l|x'|}}{|x'|^{(n-2)/2}}+\mathcal{O}(\frac{1}{|x'|^{n/2}}),\quad \theta=\frac{x'}{|x'|},
\]
as $|x'|\to \infty$, 
see \cite{Col_Kress_book, Odell_2006}.  Thus,
\[
\frac{x'}{|x'|}\cdot \nabla_{x'} u_l(x')=a_l(\theta) ik_l\frac{e^{ik_l|x'|}}{|x'|^{(n-2)/2}}+\mathcal{O}(\frac{1}{|x'|^{n/2}}),
\]
as $|x'|\to \infty$,
and therefore,
\begin{align*}
I_1&=\frac{L}{2}\sum_{1\le l< kL/\pi} \int_{|x'|=R} \bigg( \frac{|a_l(\theta)|^2 ik_l}{|x'|^{n-2}}+\mathcal{O}(\frac{1}{|x'|^{n-1}})\bigg)dS(x')\\
&=\frac{L}{2}\sum_{1\le l< kL/\pi} \int_{|x'|=1} \bigg( |a_l(\theta)|^2 ik_l+\mathcal{O}(\frac{1}{R})\bigg)dS(x').
\end{align*}
Letting $R\to \infty$ in \eqref{eq_im_part_1}, we obtain that 
\[
\sum_{1\le l< kL/\pi} \int_{|x'|=1}  |a_l(\theta)|^2 k_l dS(x')=0.
\]
Hence, $a_l\equiv 0$ for all $1\le l< kL/\pi$. By Rellich's theorem, $u_l=0$, $1\le l< kL/\pi$,  for $|x'|>R$, see \cite{Horm_1973}. Thus, $u=u_0$ for $|x'|>R$, and therefore, by 
\eqref{eq_resolv_p}, $u\in L^2(\Sigma)$.  Since $k^2$ is not an eigenvalue of $\mathcal{L}_{A,q}$, we conclude that  $u=0$ in $\Sigma$. The proof is complete.

\end{proof}

\begin{rem}

Let $A\in W^{1,\infty}(\Sigma,\R^n)\cap \mathcal{E}'(\overline{\Sigma},\R^n)$ and  $q\in L^\infty(\Sigma,\R)\cap \mathcal{E}'(\overline{\Sigma},\R)$. Assume that  $k\ge \pi/L$ is such that $k^2$ is not an eigenvalue of the operator $\mathcal{L}_{A,q}$ and  $k\ne \pi l/L$, for all $l=1,2,\dots$. 
Then it follows from the arguments in the proof of Proposition \emph{\ref{prop_admiss_fr}} that $k$ is admissible for $\mathcal{L}_{-A,q}$.  
\end{rem}

Let $R>0$ be such that  $\mathcal{L}_{A,q}=-\Delta$ in $\Sigma_{>R}$.  Let $S>R$. The operator $\mathcal{L}_{A,q}$ in $L^2(\Sigma_{<S})$, equipped with the domain $H^2(\Sigma_{<S})\cap H^1_0(\Sigma_{<S})$, which we shall denote by $\mathcal{L}_{A,q}^D$,  is closed with discrete spectrum. Let $z\in\C$ be such that $\Im z\ne 0$ and $z$ is in the resolvent set of $\mathcal{L}_{A,q}$. 
Let $\phi\in C^\infty(\overline{\Sigma_{<S}})$ be compactly supported, $0\le \phi\le 1$, and such that $\phi=1$ in $\overline{\Sigma_{<R}}$.

\begin{prop} 
\label{prop_Lax-Phil}
Let $k\ge \pi/L$ and let the assumptions \emph{(A.I)}, \emph{(A.II)} be satisfied. Then  for any 
$F\in L^2(\Sigma_{<S})$, there exists a unique $g\in L^2(\Sigma_{<S})$
such that 
\begin{equation}
\label{eq_u_in_prop}
u=\phi(\mathcal{L}_{A,q}^D-z)^{-1} g+ (1-\phi)R_0(k) g
\end{equation}
is the admissible solution of the  Dirichlet problem \eqref{eq_Dirichlet_app_first} 
in the sense of Definition \emph{\ref{def_adm}}. 
\end{prop}

\begin{proof} 
Applying the operator $\mathcal{L}_{A,q}-k^2$ to $u$ in  \eqref{eq_u_in_prop}, we get
\[
(\mathcal{L}_{A,q}-k^2)u=g+Tg,
\]
where
\begin{equation}
\label{eq_T_g}
Tg:=\phi(z-k^2)(\mathcal{L}_{A,q}^D-z)^{-1} g+[\mathcal{L}_{A,q},\phi]((\mathcal{L}_{A,q}^D-z)^{-1} g-R_0(k)g).
\end{equation}
Given $F\in L^2(\Sigma_{<S})$, we would like to find $g\in L^2(\Sigma_{<S})$ such that 
\begin{equation}
\label{eq_fred}
g+Tg=F. 
\end{equation}

Let us first check that 
the operator 
\[
T:  L^2(\Sigma_{<S})\to L^2(\Sigma_{<S})
\]
is compact. Indeed, we have
\begin{align*}
(\mathcal{L}_{A,q}^D-z)^{-1}&: L^2(\Sigma_{<S})\to  H^2(\Sigma_{<S})\cap H^1_0(\Sigma_{<S})\hookrightarrow L^2(\Sigma_{<S}),\\
R_0(k)&: L^2(\Sigma_{<S})\to H^2_{\textrm{loc}}(\overline{\Sigma}).
\end{align*}
Now  the commutator  is given by
\[
[\mathcal{L}_{A,q},\phi]=-2\nabla\phi\cdot\nabla-\Delta\phi+2A\cdot D\phi,
\]
and 
we get
\begin{align*}
[\mathcal{L}_{A,q},\phi]&: H^2(\Sigma_{<S})\to H^1(\Sigma_{<S})\hookrightarrow L^2(\Sigma_{<S}),\\
[\mathcal{L}_{A,q},\phi]& :H^2_{\textrm{loc}}(\overline{\Sigma})\to H^1(\Sigma_{<S})\hookrightarrow L^2(\Sigma_{<S}),
\end{align*}
which show the compactness of the operator $T$.  

Hence, the operator $I+T$ is Fredholm of index zero and therefore, to show that  the equation \eqref{eq_fred} has a unique solution, it suffices to check that  $F=0$ implies that $g=0$. 

Assume that $F=0$. Then the assumption (A.II) implies that $u=0$ in $\Sigma$. 
Let $u_1:=(\mathcal{L}_{A,q}^D-z)^{-1} g$ and $u_2:=R_0(k)g$. Then 
\begin{equation}
\label{eq_u_0}
\phi u_1+(1-\phi)u_2=0.
\end{equation}

Let us first  consider the set
\[
\Sigma_1:=\{x\in \Sigma_{<S}: \phi(x)=1\}.
\]
We have $u_1=0$ in $\Sigma_1$, and  it follows from \eqref{eq_T_g} that $Tg=0$ in $\Sigma_1$. 
Hence, \eqref{eq_fred} implies that $g=0$ in $\Sigma_1$.

Consider now the set
\[
\Sigma_1^c:=\{x\in \Sigma_{<S}: \phi(x)\ne 1\}.
\]
It follows from \eqref{eq_u_0} that $u_2=\phi(u_2-u_1)$.  We have
\[
(-\Delta-z)(u_2-u_1)=(k^2-z)\phi(u_2-u_1),
\]
and furthermore, $u_1|_{\p\Sigma_{<S}}=u_2|_{\p\Sigma_{<S}}=0$.  Thus,
\begin{equation}
\label{eq_imp_id}
\begin{aligned}
\int_{\Sigma_{<S}} (k^2-z)\phi|u_2-u_1|^2dx &= \int_{\Sigma_{<S}}(-\Delta-z)(u_2-u_1)(\bar u_2-\bar u_1 )dx\\
&= \int_{\Sigma_{<S}}(|\nabla(u_2-u_1)|^2-z|u_2-u_1|^2)dx. 
\end{aligned}
\end{equation}
Taking the imaginary part in \eqref{eq_imp_id}, we obtain that 
\[
\int_{\Sigma_{<S}} (1-\phi)|u_2-u_1|^2dx=0. 
\]
Hence, $u_2-u_1=0$ in  $\Sigma_1^c$. Thus, \eqref{eq_u_0} implies 
that $\phi u_1=0$ in  $\Sigma_1^c$.  It follows from
\eqref{eq_T_g}  that $Tg=(z-k^2)\phi u_1=0$  in  $\Sigma_1^c$, and therefore, $g=0$ in $\Sigma_1^c$.  The proof is complete. 

\end{proof}

Let now  $k\ge 0$. It will be convenient to have the following definition.

\begin{defn} 

\label{def_adm_k}
A frequency  $k\ge 0$ is said to be admissible for the operator $\mathcal{L}_{A,q}$, if the following holds:

\begin{itemize}

\item[(i)] if $k<\pi/L$, then $k^2$ does not belong to the discrete spectrum of the operator $\mathcal{L}_{A,q}$, equipped with the domain $H^1_0(\Sigma)\cap H^2(\Sigma)$;

\item[(ii)] if $k\ge \pi/L$, then the assumptions \emph{(A.I)} and \emph{(A.II)} are fulfilled.  

\end{itemize}

\end{defn}

\begin{defn}

\label{defn_add_sol_app_2}

Let $k\ge 0$ be admissible and $F\in L^2(\Sigma)\cap\mathcal{E}'(\overline{\Sigma})$. Then a solution $u$ of the problem \eqref{eq_Dirichlet_app_first} is said to be admissible, if the following holds:

\begin{itemize}

\item[(i)] if $k<\pi/L$, then $u\in H^2(\Sigma)$ is the unique solution given by Corollary \emph{\ref{cor_disc_spec_sol}};  
\item[(ii)] if $k\ge \pi/L$, then $u$ is the admissible solution in the sense of Definition \emph{\ref{def_adm}}.
\end{itemize}

\end{defn}

Consider the following Dirichlet problem
\begin{equation}
\label{eq_final_ap}
\begin{aligned}
(\mathcal{L}_{A,q}(x,D)-k^2)u(x)&=0\quad \textrm{in}\quad \Sigma,\\
u&=f\quad\textrm{on}\quad \Gamma_1,\\
u&=0\quad\textrm{on}\quad \Gamma_2,
\end{aligned}
\end{equation}
where $k\ge 0$ is fixed admissible, and  $f\in H^{3/2}(\Gamma_1)\cap \mathcal{E}'(\Gamma_1)$. Let $F\in H^2(\Sigma)\cap \mathcal{E}'(\overline{\Sigma})$ be such that 
$F|_{\Gamma_1}=f$ and $F|_{\Gamma_2}=0$. 
We solve the problem \eqref{eq_final_ap} by setting
\[
u=F+u_0,
\]
where $u_0$ is the admissible solution of the problem,
\begin{align*}
(\mathcal{L}_{A,q}(x,D)-k^2)u_0&=(k^2-\mathcal{L}_{A,q}(x,D))F\quad \textrm{in}\quad \Sigma,\\
u_0&=0\quad\textrm{on}\quad \p \Sigma,
\end{align*}
in the sense of Definition \ref{defn_add_sol_app_2}.  We have $u\in H^2_{\textrm{loc}}(\overline{\Sigma})$.  We shall refer to this solution $u$ of the problem \eqref{eq_final_ap} as the admissible solution.

In the main text we shall have to use the following Green's formula in the infinite slab $\Sigma$. 

\begin{prop}

\label{prop_Greens_app}

Let $k\ge 0$ be an admissible frequency for $\mathcal{L}_{A,q}$ and $\mathcal{L}_{-A,q}$, let $u$ be the admissible solution to the problem \eqref{eq_final_ap} with some $f\in H^{3/2}(\Gamma_1)\cap \mathcal{E}'(\Gamma_1)$,  and $\overline{v}$ be the admissible solution of the problem
\begin{align*}
(\mathcal{L}_{-A,q}-k^2)\overline{v}&=g, \quad \textrm{in}\quad \Sigma,\\
\overline{v}|_{\p \Sigma}&=0,
\end{align*}
for some $g\in L^2(\Sigma)\cap \mathcal{E}'(\overline{\Sigma})$. Then we have
\[
(\mathcal{L}_{\overline{A},\overline{q}} v,u)_{L^2(\Sigma)}-(v,\mathcal{L}_{A,q} u)_{L^2(\Sigma)}=-(\p_\nu v, u)_{L^2(\Gamma_1)}. 
\]
Here $\nu$ is the unit outer normal to $\Gamma_1$. 
\end{prop}

\begin{proof} First notice that $v$ satisfies $(\mathcal{L}_{\overline{A},\overline{q}}-k^2)v=\overline{g}$ in $\Sigma$. 
Let $R>0$ be such that $\supp(A^{(j)})\subset \overline{\Sigma_{<R}}$. Setting
\[
\Gamma_j\cap \p \Sigma_{<R}=d_j(R), \quad j=1,2,\quad \p \Sigma_{<R}\cap \Sigma=d_3(R),
\]
 we have $A^{(j)}=0$ on $d_3(R)$. By \eqref{eq_Green}, we get
 \begin{align*}
(\mathcal{L}_{\overline{A},\overline{q}}v,u)_{L^2(\Sigma_{<R})}-(v,\mathcal{L}_{A,q}u)_{L^2(\Sigma_{<R})}
=-(\p_\nu v,u)_{L^2(d_1(R))}\\
+
(v,\p_\nu u)_{L^2(d_3(R))}
-(\p_\nu v,u)_{L^2(d_3(R))}.
 \end{align*}
 
 We have to show that $(v,\p_\nu u)_{L^2(d_3(R))}
-(\p_\nu v,u)_{L^2(d_3(R))}$ tends to zero as $R\to \infty$. 
 Consider the case $k\ge \pi/L$ for the maximum of generality. 
 Let us write $u=u_0+u_1$, where
 \[
 u_0(x)=\sum_{l>kL/\pi}u_l(x')\sin\frac{l\pi x_n}{L},\quad  u_1(x)=\sum_{1\le l\le kL/\pi}u_l(x')\sin\frac{l\pi x_n}{L},
 \]
 and similarly, $v=v_0+v_1$.  We set
 \begin{align*}
 \int_{|x'|=R}\int_0^L (u\overline{\p_{\nu} v}-\overline{v}\p_{\nu}u)dx_ndS(x')=I_1+I_2,
 \end{align*}
 where
 \[
 I_1=\int_{|x'|=R}\int_0^L (u_0\overline{\p_{\nu} v_0}-\overline{v_0}\p_{\nu}u_0)dx_ndS(x')=\mathcal{O}(R^{-n-2}),
 \]
 as $R\to \infty$, in view of \eqref{eq_resolv_p} for $u_0$ and $v_0$.  Here 
 \begin{align*}
 I_2&=\int_{|x'|=R}\int_0^L (u_1\overline{\p_{\nu} v_1}-\overline{v_1}\p_{\nu}u_1)dx_ndS(x')\\
 &=
 \frac{L}{2}\sum_{1\le l\le kL/\pi}\int_{|x'|=R}(u_l(x')\overline{\p_{\nu} v_l(x')}-\overline{v_l(x')}\p_{\nu}u_l(x'))dS(x'). 
 \end{align*}
 Using the fact that $\p_{\nu}=(x'/R)\cdot\nabla_{x'}$ along $|x'|=R$ together with  
\eqref{eq_radiation_Som}, for $l<kL/\pi$, we get
\[
\int_{|x'|=R}(u_l(x')\overline{\p_{\nu} v_l(x')}-\overline{v_l(x')}\p_{\nu}u_l(x'))dS(x')=o(R^{-(n-2)})\int_{|x'|=R}dS(x')=o(1),
\]
as $R\to \infty$. 
Finally, if $k$ is such that $l=kL/\pi$ and $n\ge 4$, using  \eqref{eq_radiation_Som_2}, we obtain that
\[
\int_{|x'|=R}(u_l(x')\frac{x'}{R}\cdot \nabla_{x'}\overline{v_l(x')}-\overline{v_l(x')}\frac{x'}{R}\cdot \nabla_{x'} u_l(x'))dS(x')=\mathcal{O}(R^{3-n}),
\]
as $R\to \infty$.
The proof is complete.

\end{proof}

\end{appendix}

\section*{Acknowledgements}  

K.K. is grateful to Pavel Exner for a helpful discussion.  The research of K.K. is supported by the
Academy of Finland (project 125599).  The research of M.L. 
is partially supported 
 by the Academy of Finland Center of Excellence programme 213476. The research of
G.U. is partially supported by the National Science Foundation.

\end{document}